\UseRawInputEncoding
\RequirePackage{fix-cm}
\documentclass[smallextended]{svjour3}       
\smartqed  
\usepackage{graphicx}
\usepackage {color}
\usepackage {tikz}

\usepackage{graphicx}
\usepackage{amsmath}
\usepackage{amssymb}
\usepackage{mathtools}

\usepackage{cases}

\usepackage{framed}
\usepackage{makecell}
\usepackage{float}
\usepackage{bm}
\usepackage{multirow}
\renewcommand{\thetable}{\arabic{table}}

\newtheorem{rem}{Remark}

\begin{document}

\title{On the proximal point algorithms for solving the monotone inclusion problem\thanks{This research was supported by the National Natural Science Foundation of China under
grant 12171021, and the Fundamental Research Funds for the Central Universities.}}


\titlerunning{Proximal point algorithms}        

\author{Tao Zhang \and Shiru Li \and Yong Xia 
}

\institute{Tao Zhang \and Shiru Li \and Yong Xia (corresponding author) \at School of Mathematical Sciences, Beihang University, Beijing 100191, People's Republic of China.
	\email{\{shuxuekuangwu,\ lishiru,\ yxia\}@buaa.edu.cn}
	  }
\date{Received: date / Accepted: date}
\maketitle


\begin{abstract}
We consider finding a zero point of the maximally monotone operator $T$. First, instead of using the proximal point algorithm (PPA) for this purpose, we  employ PPA to solve its Yosida regularization $T_{\lambda}$. Then, based on an $O(a_{k+1})$ ($a_{k+1}\geq \varepsilon>0$) resolvent index of $T$, it turns out that we can establish a convergence rate of $O (1/{\sqrt{\sum_{i=0}^{k}a_{i+1}^2}})$ for both the $\|T_{\lambda}(\cdot)\|$ and the gap function $\mathtt{Gap}(\cdot)$ in the non-ergodic sense, and   $O(1/\sum_{i=0}^{k}a_{i+1})$ for $\mathtt{Gap}(\cdot)$ in the ergodic sense. Second, to enhance the convergence rate of the newly-proposed PPA,  we introduce an accelerated variant called the Contracting PPA. By utilizing a resolvent index of $T$ bounded by $O(a_{k+1})$ ($a_{k+1}\geq \varepsilon>0$), we establish a convergence rate of $O(1/\sum_{i=0}^{k}a_{i+1})$ for both $\|T_{\lambda}(\cdot)\|$ and
  $\mathtt {Gap}(\cdot)$, considering the non-ergodic sense.
Third, to mitigate the limitation that the Contracting PPA lacks a convergence guarantee,
 we propose two additional versions of the algorithm.  These novel approaches not only ensure guaranteed convergence but also provide sublinear and linear convergence rates for both $\|T_{\lambda}(\cdot)\|$ and  $\mathtt {Gap}(\cdot)$, respectively, in the non-ergodic sense.


 \end{abstract}

\section{Introduction}
Let $\mathcal{H}$ be a real Hilbert space equipped with an inner product $\langle \cdot,\cdot \rangle$ and the corresponding norm $\|\cdot\|$.
Let $T:\mathcal{H}\rightarrow {\mathcal{H}}$ be a maximally monotone operator. Assume that $T^{-1}(0)$ is nonempty and bounded. We aim to solve the monotone inclusion problem:
\begin{equation}\label{B}
{\rm find} ~x\in \mathcal{H} ~{\rm such ~that} ~0\in Tx.
\end{equation}

The resolvent operator of $T$ is defined as
\[
	J_{\lambda T}=(I+\lambda T)^{-1},
\]
where the key parameter $\lambda>0$ is known as the resolvent index of $T$. The corresponding Yosida regularization of $T$ is then defined as
 \begin{eqnarray}\label{BY}
 	T_{\lambda}=\frac{1}{\lambda}(I-J_{\lambda T}).
 \end{eqnarray}
It is well-known that $T_{\lambda}$ is $\lambda$-cocoercive, i.e.,
 \[
 \langle T_{\lambda}(x)-T_{\lambda}(y),x-y\rangle\geq \lambda\|T_{\lambda}(x)-T_{\lambda}(y)\|^2,~~\forall x,y\in\mathcal{H},
 \]
which implies that $T_{\lambda}$ is monotone and continuous. For more properties about maximally monotone operators and Yosida regularization, we refer to \cite{Bauschke2011Convex}.

Now we can equivalently rewrite problem  \eqref{B} as
 \begin{equation}\label{BB}
 {\rm find} ~x\in \mathcal{H} ~{\rm such ~that} ~0= T_{\lambda}(x).
 \end{equation}
Based on the  monotonicity and continuity of $T_{\lambda}$, problem \eqref{BB}  is also equivalent to
 \begin{equation}\label{BBB}
 {\rm find} ~x\in \mathcal{H} ~{\rm such ~that} ~\langle  y-x, T_{\lambda}(y)   \rangle\geq0,~\forall y\in\mathcal{H}.
 \end{equation}
By utilizing these two formulations \eqref{BB} and \eqref{BBB}, we introduce the following two metrics
to measure the quality of the approximate solution of \eqref{B}:
 \begin{itemize}
 	\item $\|T_{\lambda}(x)\|$;
 	\item $\mathtt {Gap}(x):=\max\limits_{y\in\mathbb{B}(x_0,D)}\{ \langle  x-y, T_{\lambda}(y)   \rangle\}$,
 \end{itemize}
 where  $\mathbb{B}(x_0,D):=\{ y\in \mathcal{H}:\|y-x_0\|\leq D\}$ and $x_0\in \mathcal H$.
Clearly, $x\in\mathcal{H}$ is a solution of \eqref{B} if and only if  $\|T_{\lambda}(x)\|=0$. Moreover, as shown in \cite[Lemma 1]{nesterov2007dual}, if  $\mathtt {Gap}(x)=0$ and $\|x_0-x\|<D$, then $x$ is a solution of \eqref{BBB}; If $x$ is a solution to \eqref{BBB} satisfying $\|x_0-x \|\leq D$, then $\mathtt {Gap}(x)=0$.

\subsection{Prior works}
A typical application of problem \eqref{B} is when we set $T=\partial f$, where $f$ is a closed, proper, and convex function in $\mathcal{H}$ and $\partial f$ represents the subgradient of $f$. In this case, problem \eqref{B}
is equivalent to  the  convex optimization problem:
\begin{equation}\label{f}
\min\limits_{x\in \mathcal{H}} f(x).
	\end{equation}
The proximal point algorithm (PPA), which was originally introduced by Martinet \cite{martinet1970regularisation}, is a fundamental iterative procedure for solving \eqref{f}:
\begin{equation}\label{PPAf}
x_{k+1}=J_{a_{k+1} \partial f}(x_k)=\arg\min\limits_{x}\left\{f(x)+\frac{1}{{{2}}a_{k+1}}\|x-x_k\|^2\right\},~a_{k+1}>0.
\end{equation}
Starting from any initial point $x_0$ satisfying $\|x_0-x_*\|\leq R$,
the following convergence rate of PPA \eqref{PPAf},
$$
f(x_{k+1})-f(x_*)\leq \frac{R^2}{2\sum_{i=0}^{k}a_{i+1}},
$$
was first established by G\"{u}ler \cite[Theorem 2.1]{guler1991convergence}.
Taylor et al. \cite[Theorem 4.1]{taylor2017exact} improved this result to
$$
f(x_{k+1})-f(x_*)\leq \frac{R^2}{4\sum_{i=0}^{k}a_{i+1}}.
$$
To accelerate the  convergence  rate of PPA \eqref{PPAf}, Doikov and Nesterov \cite{doi:10.1137/19M130769X}
introduced the Contracting  proximal  method (CPM) for solving problem \eqref{f}:
\begin{equation}\label{CPM}\tag{CPM}
\begin{cases}
A_{k+1}=a_{k+1}+A_k,~~a_{k+1}>0,\\
\hat{x}_{k+1}=\arg\min\limits_{x}\left\{
A_{k+1}f\left(
\frac{a_{k+1}x+A_kx_k}{A_{k+1}}
\right)+\frac{1}{2\beta}\|x-\hat x_k\|^2\right\},\\
x_{k+1}=\frac{a_{k+1}\hat x_{k+1}+A_kx_k}{A_{k+1}}.
\end{cases}
\end{equation}
The inspiration for this approach came from Nesterov's accelerated gradient method \cite{nesterov1983method} (see \cite{luo2022} for an explanation of second-order differential equations). The sequence generated by \ref{CPM}  has a convergence rate such that
$$f(x_{k+1})-f(x_*)\leq O(1/A_{k+1}).$$
Based on simple calculations, we can rewrite \ref{CPM} as:
\begin{equation}\label{CPPA1}
\begin{cases}
y_k=x_k+\frac{a_{k+1}A_{k-1}}{a_{k}A_{k+1}}(x_{k}-x_{k-1}),\\
{x}_{k+1}=\arg\min\limits_{x}\left\{
f\left(
x
\right)+\frac{A_{k+1}}{2\beta a_{k+1}^2}\|x-y_k\|^2\right\},
\end{cases}
\end{equation}
which can be considered as an inertial type algorithm. In fact,  \ref{CPM} includes  the accelerated PPA proposed by G\"{u}ler \cite{guler1992new} as a special case.  It is important to note that if we set
$A_{k+1}=a_{k+1}^2$ in \eqref{CPPA1} (i.e., fixing the proximal index $A_{k+1}/(2\beta a_{k+1}^2)$), it can be easily verified that $A_{k+1}=O(k^2)$.  This implies that
the convergence rate of the  Contracting  proximal  method  \eqref{CPPA1} with a fixed proximal index is $O(1/k^2)$ compared to PPA \eqref{PPAf} with a fixed proximal index, which has a convergence rate of  $O(1/k)$. Table \ref{tab1} provides a clear comparison between PPA \eqref{PPAf} and  \ref{CPM} (or \eqref{CPPA1}).
\begin{table}
	\caption{PPA and its accelerated version \eqref{CPM}  for solving problem \eqref{f}}\label{tab1}
	\begin{center}\footnotesize
		\renewcommand{\arraystretch}{1.3}
		\begin{tabular}{ccc}\hline
			&
			Resolvent index of $\partial f$ & 	
			Convergence  rate \\		\hline
			PPA \eqref{PPAf} & $O(a_{k+1})$ &{$O\left({1}/{\sum_{i=0}^{k}}a_{i+1} \right)$}	\\ \hline
			\ref{CPM}  \cite{doi:10.1137/19M130769X}  &$O\left({a_{k+1}^2}/{A_{k+1}}\right) $&$O\left({1}/{\sum_{i=0}^{k}}a_{i+1} \right)$\\		\hline
		\end{tabular}
	\end{center}
\end{table}

The Contracting proximal method for solving problem \eqref{f} will serve as the primary motivation behind the extended approaches presented in this paper, aimed at accelerating the PPA for solving problem \eqref{B}.

Rockafellar \cite{rockafellar1976monotone} generalized PPA to solve the monotone inclusion problem \eqref{B}:
 \begin{equation}\label{PPA}
 x_{k+1}=J_{a_{k+1} T}(x_k).
 \end{equation}
Nowadays, many well-known  algorithms can  be considered as PPA for solving maximally monotone operators. These include  augmented Lagrangian method,  the proximal method of multipliers \cite{rockafellar1976augmented}, the Douglas-Rachford splitting method \cite{douglas1956numerical,lions1979splitting}, the alternating direction method of multipliers  \cite{gabay1976dual,glowinski1975}, and the primal-dual hybrid gradient method \cite{chambolle2011first,chambolle2016ergodic}.

For the convergence rate of PPA for solving problem \eqref{B} with $a_k\equiv\lambda>0$ ($k\geq 0$),
it was shown in \cite[Proposition 8]{brezis1978produits} that
\[
\|x_{k+1}-x_k\|\leq  \frac{R}{\sqrt{k+1}} ~~{\rm or}~~\|T_{\lambda}(x_k)\|\leq
\frac{R}{\sqrt{k+1}}.
\]
Recently,  Gu and Yang \cite{gu2020tight} extended this result to
 $$
 \|x_{k+1}-x_k\|\leq \sqrt{\Big(1-\frac{1}{k+1}\Big)^{k}\frac{R^2}{k+1}}~~{\rm or}~~\|T_{\lambda}(x_k)\|\leq \sqrt{\Big(1-\frac{1}{k+1}\Big)^{k}\frac{R^2}{k+1}},
 $$
when dim$\mathcal{H}\geq2$.
The convergence rate of PPA \eqref{PPA} with time-varying values of $a_k$ was established by Dong \cite{dong2014proximal}:
\begin{equation*}
\|T_{a_{k+1}}(x_{k+1})\|\leq O\left(1/{\sqrt{\sum_{i=0}^{k}a_{i+1}^2}}\right).
	\end{equation*}

In order to accelerate the convergence rate of PPA for solving problem \eqref{B},
Kim \cite{kim2021accelerated} introduced the following iterative procedure:
 \begin{equation}\label{APPA}\tag{Accelerated PPA}
 	\begin{cases}
 	x_{k+1}=J_{\lambda T}(z_k),\\
 	z_{k+1}=x_{k+1}+\frac{k}{k+2}(x_{k+1}-x_k)+\frac{k}{k+2}(z_{k-1}-x_k),
 	\end{cases}
 \end{equation}
 and showed that the generated sequence satisfies that
 \[
 \|T_{\lambda}(x_k)\|\leq O(1/k).
 \]
It was shown in \cite[Proposition 4.3]{contreras2023optimal} and \cite[Chapter 12.2, Theorem 18]{ryu2022large} that   \ref{APPA} is equivalent to the Halpern iteration \cite{halpern1967fixed} due to Lieder \cite{lieder2021convergence}.

Motivated by the second-order dynamical system,   Attouch and Peypouquet  \cite{2017Convergence} presented the following Regularized Inertial Proximal  Algorithm (RIPA) for solving problem \eqref{B}:
 \begin{equation}\label{RIPA}\tag{RIPA}
 	\begin{cases}
 	\lambda_k=(1+\epsilon)\frac{s}{\alpha^2}k^2,~s>0,~\epsilon>0,~\alpha>2,\\
 	z_k=x_k+\left(1-\frac{\alpha}{k}\right)(x_k-x_{k-1}),\\
 	x_{k+1}=\frac{\lambda_k}{\lambda_k+s}z_k+\frac{s}{\lambda_k+s}J_{(\lambda_k+s)T}(z_k),
 	\end{cases}
 \end{equation}
 and showed that
 \[
 \|T_{\lambda_k+s}(x_k)\|=o(1/k).
 \]
Another second-order dynamical system based algorithm, proposed by Attouch and L\'{a}szl\'{o} \cite{attouch2020newton},  is the following Proximal Regularized Inertial Newton Algorithm (PRINA):
 \begin{equation}\label{PRINA}\tag{PRINA}
 \begin{cases}
\alpha_k=\frac{rk+q-1}{r(k+1)+q},~\lambda_k=\lambda k^2,~r>0,~\lambda>\frac{(2\beta+s)^2r^2}{s},~q\in \mathbb{R},\\
 z_{k-1}=\left(1-\beta\left(
 \frac{1}{\lambda_k}-\frac{1}{\lambda_{k-1}}
 \right)\right)x_k+\left(\alpha_k-\frac{\beta}{\lambda_{k-1}}\right)(x_k-x_{k-1})\\~~~~~~~~~+\beta\left(
 \frac{1}{\lambda_{k}}J_{\lambda_kT}(x_k)-\frac{1}{\lambda_{k-1}}J_{\lambda_{k-1}T}(x_{k-1})
 \right),\\
x_{k+1}= \frac{\lambda_{k+1}}{\lambda_{k+1}+s}z_{k-1}+\frac{s}{\lambda_{k+1}+s}J_{(\lambda_{k+1}+s)T}(z_{k-1}).
 \end{cases}
 \end{equation}
 This algorithm generates a sequence that satisfies
 \[
 \|T_{\lambda_k}(x_k)\|=o(1/k^2).
 \]
 Recently, Maing\'{e} \cite{mainge2021accelerated} introduced the following Corrected Relaxed Inertial Proximal Algorithm (CRIPA):
 \begin{equation}\label{CRIPA}\tag{CRIPA}
 	\begin{cases}
 	\kappa_k=\kappa_{k-1}\left(
 	1+\frac{a}{bk+c}\right),~\theta_k=1-\frac{a_1}{bk+c},~\gamma_k=1-\frac{a_2}{bk+c},\\
 	z_k=x_k+\theta_k(x_k-x_{k-1})+\gamma_k(z_{k-1}-x_k),\\
 	x_{k+1}=\frac{1}{1+\kappa_k}z_k+\frac{\kappa_k}{1+\kappa_k}J_{\lambda(1+\kappa_k)T}(z_k).
 	\end{cases}
 \end{equation}
 Under  appropriate conditions on the parameters $\{a,~b,~c,~a_1,~a_2\}$, Maing\'{e} proved that  $\kappa_k\geq k^p$ and
 \[
 	\|T_{\lambda}(x_k)\|=o(1/k^{p+1}),~\mathtt{Gap}(x_k)=o(1/k^{p+1}),~p\geq0.
 \]
Based on the second-order dynamical system, Ioan Bo\c{t} et al. \cite{bot2022fast} also  proposed an inertial PPA-type algorithm with the same convergence rate as that of \ref{CRIPA}.

We provide a summary of the comparison between the above PPA-type algorithms for solving problem \eqref{B} in the first six rows of Table \ref{tab3}.

 \begin{table}
  \setlength\tabcolsep{2pt}
  	\caption{PPA-type algorithms  for solving problem \eqref{B}}
  	\begin{center}
  		\begin{tabular}{cccccc}\hline\label{tab3}
  			&
  			\makecell{Resolvent\\index of $T$}&	
  			\makecell{Rate of\\$\|T_{\lambda}(\cdot)\|$}&	\makecell{Rate of\\$\mathtt{Gap}(\cdot)$}\\		\hline
  			PPA \eqref{PPA}&{$O(a_{k+1})$}&$O\left(1/\sqrt{\sum_{i=0}^{k}a_{i+1}^2}\right)$		 &	-\\ \hline
  				Accelerated PPA   \cite{kim2021accelerated}&$O(1)$&$O(1/k)$&-\\\hline
  		{RIPA} \cite{2017Convergence}&$O(k^2)$&$o(1/k)$	&-	\\		\hline
  	PRINA  	\cite{attouch2020newton}&$O(k^2)$&$o(1/k^2)$&-\\\hline
  	CRIPA \cite{mainge2021accelerated}&$O(k^p)$&$o(1/k^{p+1})$&\makecell{$o(1/k^{p+1})$ (non-ergodic)}\\\hline
	\multirow{2}{*}{  	\makecell{{PPA \eqref{B1}}\\{\bf (this paper)}}  }&  	\multirow{2}{*}{ $O(a_{k+1})$ }  &  	\multirow{2}{*}{ $O\left(1/{\sqrt{\sum_{i=0}^{k}a_{i+1}^2}}\right)$} &   \makecell{$O\left({1}/{\sum_{i=0}^{k}}a_{i+1} \right)$ (ergodic)} \\
	\cline{4-4} &&&\makecell{$O\left(1/{\sqrt{\sum_{i=0}^{k}a_{i+1}^2}}\right)$ (non-ergodic)} \\\hline
  			\makecell{{Contracting 	PPA}\\{\bf (this paper)}}&$O(a_{k+1})$&$O\left({1}/{\sum_{i=0}^{k}}a_{i+1} \right)$	&\makecell{$O\left({1}/{\sum_{i=0}^{k}}a_{i+1} \right)$ (non-ergodic)}\\		\hline
  		\end{tabular}
  	\end{center} 	
  \end{table}

  \subsection{Contributions}
Below are the main contributions of our paper on solving problem \eqref{B} using accelerated PPA-type algorithms.
 \begin{itemize}
\item We utilize the classical PPA to solve the equivalent problem \eqref{BB} instead of the original problem \eqref{B}. With an $O(a_{k+1})$ ($a_{k+1}\geq \varepsilon>0$) resolvent index of $T$, we can establish a convergence rate of
$O\left(1/{\sqrt{\sum_{i=0}^{k}a_{i+1}^2}}\right)$ for both the $\|T_{\lambda}(\cdot)\|$ and $\mathtt{Gap}(\cdot)$ metrics in the non-ergodic sense. Additionally, we can achieve a convergence rate of  $O(1/\sum_{i=0}^{k}a_{i+1})$ for the metric
  $\mathtt{Gap}(\cdot)$ in the  ergodic sense. See Theorem \ref{thm2.1} in Section 2.

  \item Motivated by \eqref{CPM} proposed by Doikov and Nesterov for solving convex optimization problem \eqref{f}, we propose an accelerated PPA method called the {\bf Contracting PPA} in this paper to solve problem \eqref{B}.  By utilizing a resolvent index of $T$ bounded by $O(a_{k+1})$ ($a_{k+1}\geq \varepsilon>0$), we establish a convergence rate of $O(1/\sum_{i=0}^{k}a_{i+1})$ for both $\|T_{\lambda}(\cdot)\|$ and
  $\mathtt {Gap}(\cdot)$, considering the non-ergodic sense. Please refer to  Theorem \ref {thm3.3} in Section 3. Furthermore, in Theorem \ref{EQHAL}, we demonstrate the equivalence  between the Halpern iteration \cite{lieder2021convergence} and
 our Contracting PPA under certain parameter settings.

  \item
  It should be noted that the above newly-proposed Contracting PPA does not come with a convergence guarantee.
      In order to address this limitation, we introduce in Section 4 a new version of the Contracting PPA called the {\bf Sublinear Contracting PPA}, which offers guaranteed convergence. We establish the sublinear convergence of both  $\|T_{\lambda}(\cdot)\|$ and  $\mathtt {Gap}(\cdot)$,  considering the non-ergodic sense.

  \item In Section 5, we further introduce the {\bf Linear Contracting PPA}, which comes with a convergence guarantee. We demonstrate that this algorithm achieves the linear convergence in both $\|T_{\lambda}(\cdot)\|$ and
  $\mathtt {Gap}(\cdot)$, taking into account  the non-ergodic sense.

 \end{itemize}

\section{Novel PPA iteration}
Instead of using PPA \eqref{PPA} to solve problem \eqref{B}, we consider  using PPA to solve the equivalent problem \eqref{BB} and establish the convergence rate.
More generally, we allow for inaccurate solutions in each iteration. Let $x_k-g_k$ be an approximate solution of $x_k$, where $g_k$ is the error. Then we have
\begin{eqnarray*}
x_{k+1}&=&
J_{a_{k+1}T_{\lambda}}(x_k-g_k)\\
&=&\frac{\lambda}{\lambda+a_{k+1}}(x_k-g_k)+\frac{a_{k+1}}{\lambda+a_{k+1}}J_{(\lambda+a_{k+1})T}(x_k-g_k),
\end{eqnarray*}	
where the last equality holds since $T_{\lambda}$ \eqref{BY} can be calculated as
 \begin{equation}\label{B41}
 J_{\mu T_{\lambda}}=\frac{\lambda}{\lambda+\mu}I+\frac{\mu}{\lambda+\mu}J_{(\lambda+\mu)T}.
 \end{equation}
Reorganizing the above iteration of $x_{k+1}$, we can summarize the procedure as follows.

\begin{framed}
	\noindent{\bf PPA for solving problem \eqref{BB}.}
	\noindent	
{Given $\lambda>0$, $x_0\in\mathcal{H}$, $a_{k+1}\geq \varepsilon>0$, and $\delta_k\ge0$ satisfying $\sum_{k\geq0}\delta_k<+\infty$,  find $x_{k+1}$ such that
	\begin{equation}\label{B1}
		0=a_{k+1}T_{\lambda}(x_{k+1})+x_{k+1}-x_k+g_k,
	\end{equation}
for $k=0,1,\cdots$, where $g_k\in\mathcal{H}$  satisfies  that $\|g_k\|\leq\delta_k$.}
\end{framed}
 Now, let us analyze its convergence rate.
\begin{theorem}\label{thm2.1}
	Suppose $x_*\in T^{-1}(0)$. The iterative sequence $\{x_k\}$ generated by
	 the PPA \eqref{B1} satisfies that
	 \begin{itemize}
	 	\item[\textnormal{({1})}] $\{x_k\}$ is bounded.
	 		\item[\textnormal{({2})}]
	 	$\sum_{i=0}^{+\infty} \|x_{i+1}-x_i\|^2<+\infty$.
	 		\item[\textnormal{({3})}]  $\|T_{\lambda}(x_{k+1})\|=o(1/a_{k+1})$ and $\mathtt {Gap}({x}_{k+1})=o(1/a_{k+1})$.
	 			\item[\textnormal{({4})}]
	 		$\mathtt {Gap}(\widetilde{x}_{k+1})\leq O(1/\sum_{i=0}^{k}a_{i+1})$, where $\widetilde{x}_{k+1}=\frac{1}{\sum_{i=0}^{k}a_{i+1}}\sum_{i=0}^{k}a_{i+1}x_{i+1}$.
	 			\item[\textnormal{({5})}] $\{x_k\}$
 converges weakly to a solution of \eqref{B}.

 	\item[\textnormal{({6})}]  In the case when $\delta_k\equiv0$, it holds that $\|T_{\lambda}(x_{k+1})\|\leq O\left(1/{\sqrt{\sum_{i=0}^{k}a_{i+1}^2}}\right)$ and $\mathtt {Gap}({x}_{k+1})\leq O\left(1/{\sqrt{\sum_{i=0}^{k}a_{i+1}^2}}\right)$.
	 \end{itemize}
\end{theorem}
\begin{proof}
	Based on the iteration \eqref{B1}, we can derive the following result:
	\begin{eqnarray*}
	0&=&	\langle x-x_{k+1}, a_{k+1} T_{\lambda}(x_{k+1})+x_{k+1}-x_k +g_k  \rangle\\&=& a_{k+1}	\langle x-x_{k+1},  T_{\lambda}(x_{k+1})  \rangle+	\langle x-x_{k+1},x_{k+1}-x_k  \rangle+	\langle x-x_{k+1},  g_k  \rangle\\&=&a_{k+1}	\langle x-x_{k+1},  T_{\lambda}(x_{k+1})  \rangle+	\langle x-x_{k+1},  g_k  \rangle\\&& -\frac{1}{2}\left(
	\|x_{k+1}-x\|^2-\|x_k-x\|^2+\|x_{k+1}-x_k\|^2
	\right)	.
	\end{eqnarray*}
	By rearranging terms in the above equality and applying the Cauchy-Schwarz inequality, we obtain:
	\begin{equation}\label{B37}
	\begin{aligned}
		&a_{k+1}	\langle x_{k+1}-x,  T_{\lambda}(x_{k+1})\rangle+\frac{1}{2}\|x_{k+1}-x\|^2\\\leq&
\frac{1}{2}\|x_{k}-x\|^2-\frac{1}{2}\|x_{k+1}-x_k\|^2+\|x-x_{k+1}\|\|  g_k  \|
\\\leq&
\frac{1}{2}\|x_{k}-x\|^2-\frac{1}{2}\|x_{k+1}-x_k\|^2+\delta_k\|x-x_{k+1}\|.
	\end{aligned}
	\end{equation}
	
\noindent Proof of \textnormal{({1})}.
By adding the above inequality from	$k=0$ to $k=t$, we have
	\begin{equation}\label{B36}
	\begin{aligned}
&\sum_{k=0}^{t}	a_{k+1}	\langle x_{k+1}-x,  T_{\lambda}(x_{k+1})\rangle+\frac{1}{2}\|x_{t+1}-x\|^2\\\leq&
\frac{1}{2}\|x_{0}-x\|^2-\frac{1}{2}\sum_{k=0}^{t}\|x_{k+1}-x_k\|^2+\sum_{k=0}^{t}\delta_k\|x-x_{k+1}\|.
	\end{aligned}	
	\end{equation}
	 Setting $x=x_*$ in the above inequality and disregarding the nonnegative terms $\sum_{k=0}^{t}	a_{k+1}	\langle x_{k+1}-x_*,  T_{\lambda}(x_{k+1})\rangle$ and $\sum_{k=0}^{t}\frac{1}{2}\|x_{k+1}-x_k\|^2$, we obtain that $\{x_k\}$ is bounded. This is achieved by utilizing Lemma \ref{L1} (see Appendix A), where we assign {$a_k=\frac{1}{\sqrt{2}}\|x_{k+1}-x_*\|$, $b_k=\sqrt{2}\delta_k$} and $c=\frac{1}{\sqrt{2}}\|x_{0}-x_*\|$.
	
	 \noindent Proof of \textnormal{({2})}.
The summability of  $ \|x_{k+1}-x_k\|^2$  can be inferred from the inequality \eqref{B36} when considering $x=x_*$, and the boundedness of the sequence $\{x_k\}$.

	  \noindent  Proof of \textnormal{({3})}. According to iteration \eqref{B1}, it holds that
	  $$
	  \|T_{\lambda}(x_{k+1})\|=\frac{1}{a_{k+1}}\|x_{k+1}-x_k+g_k\|\leq\frac{1}{a_{k+1}}(\|x_{k+1}-x_k\|+\|g_k\|).
	  $$
	  Noting that $\|x_{k+1}-x_k\|\rightarrow 0$ and $\|g_k\|\rightarrow 0$ as $k\rightarrow+\infty$, we obtain the conclusion.
	
	  \noindent According to the monotonicity of  $T_{\lambda}$, {for any $x\in\mathbb{B}(x_0,D),~D>0$}, we have
	  $$
	  \begin{aligned}
	  &	\langle x_{k+1}-x,  T_{\lambda}(x_{}) \rangle\leq	\langle x_{k+1}-x,  T_{\lambda}(x_{k+1}) \rangle\\=&\frac{1}{a_{k+1}}\langle x-x_{k+1},    x_{k+1}-x_k +g_k \rangle
	  \leq\frac{1}{a_{k+1}}\|x-x_{k+1}\|  ( \| x_{k+1}-x_k \|+\|g_k\|),
	  \end{aligned}
	  $$
where the equality follows from equation \eqref{B1}. 	  By applying the definition of $\mathtt{Gap}({x}_{k+1})$, as well as considering that $\|x_{k+1}-x_k\|\rightarrow 0$ and $\|g_k\|\rightarrow 0$ as $k\rightarrow+\infty$, we can arrive at the desired conclusion.
	
	 \noindent Proof of \textnormal{({4})}. {For any $x\in\mathbb{B}(x_0,D),~D>0$}, we can verify that
	 \begin{eqnarray*}
	 	\langle \widetilde x_{k+1}-x,  T_{\lambda}(x)\rangle&=&\frac{1}{ \sum_{i=0}^{k}	a_{i+1}}
	 \sum_{i=0}^{k}	a_{i+1}	\langle x_{i+1}-x,  T_{\lambda}(x_{})\rangle\\&	\leq& \frac{1}{ \sum_{i=0}^{k}	a_{i+1}} \sum_{i=0}^{k}	a_{i+1}	\langle x_{i+1}-x,  T_{\lambda}(x_{i+1})\rangle\\&\overset{\eqref{B36}}{\leq}&\frac{1}{ \sum_{i=0}^{k}	a_{i+1}}\Big(\frac{1}{2}\|x_{0}-x\|^2+\sum_{i=0}^{k}\delta_i\|x-x_{i+1}\|\Big).
	\end{eqnarray*}
	 According to the definition of $\mathtt{Gap}(\widetilde{x}_{k+1})$, we obtain the conclusion.
	
 \noindent Proof of \textnormal{({5})} follows from Lemma \ref{L4} (see Appendix A). It suffices to establish the validity of the two conditions stated in Lemma \ref{L4}.

First, we will prove the first condition in Lemma \ref{L4}.
By substituting $x=x_*$ into \eqref{B37} and disregarding the nonnegative terms $\langle x_{k+1}-x_*,  T_{\lambda}(x_{k+1})\rangle$ and $\frac{1}{2}\|x_{k+1}-x_k\|^2$, we can assert that
since $\{x_k\}$ is bounded, there exists a constant $C>0$ such that
  $$
  \begin{aligned}
  &\frac{1}{2}\|x_{k+1}-x_*\|^2\leq
  \frac{1}{2}\|x_{k}-x_*\|^2+\delta_kC.
  \end{aligned}
  $$
  Then $\frac{1}{2}\|x_{k+1}-x_*\|^2$ converges by invoking Lemma \ref{L3} with $a_k=\frac{1}{2}\|x_{k}-x_*\|^2$ and $d_k=\delta_kC$.

Then, we proceed to prove the second condition in Lemma \ref{L4}. Given the provided condition  ($a_{k+1}\geq \varepsilon>0$) and  $\|T_{\lambda}(x_{k+1})\|=o(1/a_{k+1})$, we can conclude that  $\lim\limits_{k\rightarrow+\infty}\|T_{\lambda}(x_k)\|=0$. Consequently, every weak cluster point of  $x_k$ belongs to the solution of \eqref{B} by utilizing Lemma \ref{L6}.

  \noindent  Proof of \textnormal{({6})}. Based on  the monotonicity of  $T_{\lambda}$ and iteration \eqref{B1}, if  $\delta_k\equiv0$,  it holds
  \begin{equation}\label{B3}
  \begin{aligned}
  &	0\leq\langle x_{k+1}-x_k, \frac{x_k-x_{k+1}}{a_{k+1}}-\frac{x_{k-1}-x_{k}}{a_{k}}\rangle\\\Longrightarrow&\frac{1}{a_{k+1}}\|x_{k+1}-x_k\|^2\leq\langle  x_{k+1}-x_k, \frac{x_{k}-x_{k-1}}{a_{k}}  \rangle\\\Longrightarrow&\frac{1}{a_{k+1}}\|x_{k+1}-x_k\|\leq \frac{1}{a_{k}}\|x_{k-1}-x_{k}\| .
  \end{aligned}
  \end{equation}
The conclusion follows from the boundedness of $\{x_k\}$ and the following inequalities:
  \begin{eqnarray*}
  &&\begin{aligned}
  +\infty>	&\sum_{i=0}^{k} \|x_{i+1}-x_i\|^2=\sum_{i=0}^{k}a_{i+1}^2 \frac{1}{a_{i+1}^2}\|x_{i+1}-x_i\|^2\\\overset{\eqref{B3}}{\geq}&\frac{1}{a_{k+1}^2}\|x_{k+1}-x_k\|^2\sum_{i=0}^{k}a_{i+1}^2=\|T_{\lambda}(x_{k+1})\|^2\sum_{i=0}^{k}a_{i+1}^2,
  \end{aligned}\\&&
 \begin{aligned}
  \langle x_{k+1}-x,  T_{\lambda}(x_{}) \rangle\leq	\langle x_{k+1}-x,  T_{\lambda}(x_{k+1}) \rangle\leq \| x_{k+1}-x\| \cdot \|  T_{\lambda}(x_{k+1})\|.
  \end{aligned}
  \end{eqnarray*}
\end{proof}

\section{{\bf Contracting PPA: acceleration}}
In this section,  we propose the  Contracting PPA as an acceleration technique for the new PPA \eqref{B1}.
This approach is inspired by \ref{CPM}, which is utilized to tackle problem \eqref{f}.
\begin{framed}
	\noindent{\bf Contracting PPA.}
	
	\noindent{
Given $a_{k+1}\geq \varepsilon>0$, $x_{-1}=x_0\in\mathcal{H}$, $A_{-1}=0$, for $k=0,1,\cdots$, set
		\begin{eqnarray}
		A_{k}&=&a_{k}+A_{k-1}, \label{B8v:0}\\
			v_k&=&\frac{A_{k}x_{k}-A_{k-1}x_{k-1}}{a_{k}}. \label{B8v}
		\end{eqnarray}
For given $\lambda>0$ and $\delta_k\ge0$ satisfying $\sum_{k\geq0}\delta_k<+\infty$,
find $x_{k+1}$ such that
	\begin{equation}\label{B8}
		0=A_{k+1}T_{\lambda}(x_{k+1})-A_{k}T_{\lambda}(x_{k})+v_{k+1}-v_k+g_k,
	\end{equation}
 where $g_k\in \mathcal{H}$  satisfies  that $\|g_k\|\leq\delta_k$. }
\end{framed}


The following lemma offers a new perspective on the  Contracting PPA.

\begin{lemma}
 The Contracting PPA is equivalent to
\begin{equation}\label{B43}
~
\left\{
\begin{array}{lcl}
y_k=x_k+\frac{a_{k+1}A_{k-1}}{a_{k}A_{k+1}}(x_{k}-x_{k-1}),\\
	z_k=y_k+\frac{a_{k+1}A_k}{A_{k+1}a_k}(z_{k-1}-x_k-\frac{a_k}{A_k}g_{k-1}),\\
	x_{k+1}=\Big(\frac{\lambda}{\lambda+a_{k+1}}I+\frac{a_{k+1}}{\lambda+a_{k+1}}J_{(\lambda+a_{k+1})T}\Big)\big(z_k-\frac{a_{k+1}}{A_{k+1}}g_k\big).
\end{array}
\right.
\end{equation}
Specifically,  the Contracting PPA with  $g_k\equiv0$ (the exact version) is equivalent to:
\begin{equation}\label{B42}
	x_{k+1}=\Big(\frac{\lambda}{\lambda+a_{k+1}}I+\frac{a_{k+1}}{\lambda+a_{k+1}}J_{(\lambda+a_{k+1})T}\Big)\Big(\frac{a_{k+1}}{A_{k+1}}s_0+\frac{A_k}{A_{k+1}}x_k\Big),
\end{equation}where $s_0=A_{0}T_{\lambda}(x_{0})+v_0.$
\end{lemma}
\begin{proof}
	According to the iteration of $v_k$ in the Contracting PPA, we have
	\begin{equation}\nonumber
	\begin{aligned}
		v_{k+1}-v_k=&\frac{A_{k+1}}{a_{k+1}}(x_{k+1}-x_k)-\frac{A_{k-1}}{a_{k}}(x_{k}-x_{k-1})\\=&
	\frac{A_{k+1}}{a_{k+1}}\Big(
	x_{k+1}-\big(\underbrace{x_k+\frac{a_{k+1}A_{k-1}}{a_{k}A_{k+1}}(x_{k}-x_{k-1})}_{y_k}\big)
	\Big).
		\end{aligned}
	\end{equation}
For the sake of convenience, let us define $z_k':=y_k+\frac{a_{k+1}A_k}{A_{k+1}}T_{\lambda}(x_{k})$. Then, it follows that \eqref{B8} in the Contracting PPA is equivalent to
	\begin{eqnarray}
			 0&=&A_{k+1}T_{\lambda}(x_{k+1})-A_{k}T_{\lambda}(x_{k})+	\frac{A_{k+1}}{a_{k+1}}(x_{k+1}-y_k)+g_k\nonumber\\\Longleftrightarrow~~~~~0&=&
		T_{\lambda}(x_{k+1})+	\frac{1}{a_{k+1}}(x_{k+1}-z_k'+\frac{a_{k+1}}{A_{k+1}}g_k)
		\label{B32}	\\\Longleftrightarrow x_{k+1}&=&(I+a_{k+1}T_{\lambda})^{-1}\big(z_k'-\frac{a_{k+1}}{A_{k+1}}g_k\big)=J_{a_{k+1}T_{\lambda}}\big(z'_k-\frac{a_{k+1}}{A_{k+1}}g_k\big)\nonumber\\&\overset{\eqref{B41}}{=}&\Big(\frac{\lambda}{\lambda+a_{k+1}}I+\frac{a_{k+1}}{\lambda+a_{k+1}}J_{(\lambda+a_{k+1})T}\Big)\big(z_k'-\frac{a_{k+1}}{A_{k+1}}g_k\big).\label{B33}
	\end{eqnarray}
Additionally, we can remove $T_{\lambda}(x_k)$ in $z_k'=y_k+\frac{a_{k+1}A_k}{A_{k+1}}T_{\lambda}(x_{k})$ by applying \eqref{B32}, which leads us to
	$$
	z_k'=y_k+\frac{a_{k+1}A_k}{A_{k+1}a_k}(z_{k-1}'-x_k-\frac{a_k}{A_k}g_{k-1}).
	$$
 Finally, by setting $z_k=z_k'$ in \eqref{B33}, we can conclude that the Contracting PPA is equivalent to the expression given by \eqref{B43}.
	
	For the exact Contracting PPA with  $g_k\equiv0$, we have
	$$
	A_{k+1}T_{\lambda}(x_{k+1})+v_{k+1}=A_{k}T_{\lambda}(x_{k})+v_k=A_{0}T_{\lambda}(x_{0})+v_0=s_0,
	$$which can be simplified to:
	$$
	\begin{aligned}
	x_{k+1}&=(I+a_{k+1}T_{\lambda})^{-1}\big(\frac{a_{k+1}}{A_{k+1}}s_0+\frac{A_k}{A_{k+1}}x_k\big)=J_{a_{k+1}T_{\lambda}}\big(\frac{a_{k+1}}{A_{k+1}}s_0+\frac{A_k}{A_{k+1}}x_k\big).
	\end{aligned}
	$$
Then, using \eqref{B41}, we can derive the remaining conclusion.
\end{proof}

Let us now consider the following {potential function}
\begin{equation}\label{B16}
\varTheta_{k+1}(x):=A_{k+1}\langle x_{k+1}-x, T_{\lambda}(x_{k+1})\rangle+\frac{1}{2}\|v_{k+1}-x\|^2.
\end{equation}
It plays a crucial role in the analysis of the convergence rate of the Contracting PPA.

\begin{lemma}\label{LC}
	Suppose $x_*\in T^{-1}(0)$. The iterative sequence $\{x_k\}$ generated by the Contracting PPA satisfies that
	\begin{equation}\label{B17}
	\begin{aligned}
	\varTheta_{k+1}(x)=	\varTheta_{k}(x)&-
	\left(
\frac{A_kA_{k+1}}{a_{k+1}}\langle x_{k+1}-x_k, T_{\lambda}(x_{k+1})-T_{\lambda}(x_{k}) \rangle\right. \\&\left.
	+\frac{1}{2}\|v_{k+1}-v_k\|^2+\langle
	x-v_{k+1}, g_k
	\rangle
	\right).
	\end{aligned}
	\end{equation}
	In addition,
	\begin{itemize}
			\item[\textnormal{({1})}] $\{v_k\}$ is bounded.
	\item[\textnormal{({2})}]	$\sum_{k\geq0}\vert\langle
			v_{k+1}-x_*, g_k
			\rangle\vert<+\infty.$
	\item[\textnormal{({3})}]	$\sum_{k\geq0}	\frac{A_kA_{k+1}}{a_{k+1}}\| T_{\lambda}(x_{k+1})-T_{\lambda}(x_{k})\|^2<+\infty.$
	\item[\textnormal{({4})}]	$\sum_{k\geq0} \|v_{k+1}-v_k\|^2<+\infty$.
	\end{itemize}
\end{lemma}
\begin{proof}
First, we can verify the following equalities hold.
\begin{eqnarray*}
&&\begin{aligned}
&{\rm(i)}\langle x-v_{k+1}, A_{k+1}T_{\lambda}(x_{k+1})\rangle
= \langle x-x_{k+1}-\frac{A_k}{a_{k+1}}(x_{k+1}-x_k), A_{k+1}T_{\lambda}(x_{k+1}) \rangle\\=&
A_{k+1}\langle x-x_{k+1}, T_{\lambda}(x_{k+1})\rangle-\frac{A_kA_{k+1}}{a_{k+1}}\langle
x_{k+1}-x_k,T_{\lambda}(x_{k+1})
\rangle,
\end{aligned}\\
&&
\begin{aligned}
&{\rm(ii)}\langle x-v_{k+1}, -A_{k}T_{\lambda}(x_{k})\rangle
= \langle x-x_{k}-\frac{A_{k+1}}{a_{k+1}}(x_{k+1}-x_k), -A_{k}T_{\lambda}(x_{k}) \rangle\\=&
-A_{k}\langle x-x_{k}, T_{\lambda}(x_{k})\rangle+\frac{A_kA_{k+1}}{a_{k+1}}\langle
x_{k+1}-x_k,T_{\lambda}(x_{k})
\rangle,
\end{aligned}\\
&&~~{\rm(iii)}
\langle x-v_{k+1}, v_{k+1}-v_k\rangle=-\frac{1}{2}(\|v_{k+1}-x\|^2-\|v_k-x\|^2+\|v_{k+1}-v_k\|^2).
\end{eqnarray*}
Then, by adding the above three equalities, we obtain the following condition:
\begin{small}
\[
\begin{aligned}
0=&\langle  x-v_{k+1},   A_{k+1}T_{\lambda}(x_{k+1})-A_{k}T_{\lambda}(x_{k})+v_{k+1}-v_k +g_k  \rangle\\=&\underbrace{A_{k+1}\langle x-x_{k+1}, T_{\lambda}(x_{k+1})\rangle\!-\!\frac{1}{2}\|v_{k+1}-x\|^2}_{-\Theta_{k+1}(x)}
\!-\!\underbrace{ (A_{k}\langle x-x_{k}, T_{\lambda}(x_{k})\rangle\!-\!\frac{1}{2}\|v_{k}-x\|^2 )}_{-\Theta_{k}(x)}\\&\!-\!
\left(
\frac{A_kA_{k+1}}{a_{k+1}}\langle x_{k+1}-x_k, T_{\lambda}(x_{k+1})-T_{\lambda}(x_{k}) \rangle
\!+\!\frac{1}{2}\|v_{k+1}-v_k\|^2\!+\!\langle x-
v_{k+1}, g_k
\rangle\right)
.
\end{aligned}
\]
\end{small}
By rearranging some terms, we can express it as  \eqref{B17}.

Subsequently, by setting $x=x_*$ in \eqref{B17} and utilizing the $\lambda$-cocoercive of $T_{\lambda}$, we can deduce that
\begin{small}
	\begin{equation}\label{B19}
	\begin{aligned}
&\frac{1}{2}\|v_{k+1}-x_*\|^2\leq	\varTheta_{k+1}(x_*)
\\\leq&\varTheta_{k}(x_*)-
	 (
	\lambda\frac{A_kA_{k+1}}{a_{k+1}}\| T_{\lambda}(x_{k+1})-T_{\lambda}(x_{k})\|^2
	+\frac{1}{2}\|v_{k+1}-v_k\|^2+\langle
	x_*-v_{k+1}, g_k
	\rangle )\\\leq&\varTheta_{k}(x_*)-
	 (
	\lambda\frac{A_kA_{k+1}}{a_{k+1}}\| T_{\lambda}(x_{k+1})-T_{\lambda}(x_{k})\|^2
	+\frac{1}{2}\|v_{k+1}-v_k\|^2 )+\|v_{k+1}-x_*\|\| g_k\|
	\\\leq&\varTheta_{0}(x_*)\!-\!
	\sum_{i=0}^{k} (
	\lambda\frac{A_iA_{i+1}}{a_{i+1}}\| T_{\lambda}(x_{i+1})\!-\!T_{\lambda}(x_{i})\|^2
	\!+\!\frac{1}{2}\|v_{i+1}-v_i\|^2 )\!+\!\sum_{i=0}^{k}\|v_{i+1}-x_*\|\delta_i.
	\end{aligned}
\end{equation}
\end{small}
By ignoring some negative terms on the right side of \eqref{B19} and applying Lemma \ref{L1} with { $a_k=\frac{1}{\sqrt{2}}\|v_{k+1}-x_*\|$, $b_k=\sqrt{2}\delta_k$, and $c=\sqrt{\varTheta_{0}(x_*)}$, we obtain
\begin{equation*}
\frac{1}{\sqrt{2}}\|v_{k+1}-x_*\|\leq \sqrt{\varTheta_{0}(x_*)}+\sqrt{2}\sum_{i=0}^{\infty}\delta_i<+\infty.
\end{equation*} }
Consequently, we can conclude that $\{v_k\}$ is bounded. Furthermore, since $\delta_k$ is summable,  the right side of \eqref{B19} is also bounded. Finally, by utilizing \eqref{B19},
we can immediately deduce the remaining conclusions.
\end{proof}

\begin{theorem}\label{thm3.3}
The iterative sequence $\{x_k\}$  generated by the Contracting PPA satisfies that
		\begin{eqnarray}
	\label{B38}	\|T_{\lambda}(x_{k+1})\|\leq O(1/A_{k+1}),\\
\label{B39}		\mathtt{Gap}(x_{k+1})\leq O(1/A_{k+1}).
		\end{eqnarray}
\end{theorem}
\begin{proof}
	Due to the iteration \eqref{B8} in the Contracting PPA, we have
	\[
	\begin{aligned}
		&\|A_{k+1}T_{\lambda}(x_{k+1})+v_{k+1}\|=\|A_{k}T_{\lambda}(x_{k})+v_k-g_k\|\\\leq&
		\|A_{k}T_{\lambda}(x_{k})+v_k\|+\|g_k\|\leq
		\|A_{0}T_{\lambda}(x_{0})+v_0\|+\sum_{i=0}^{k}\|g_i\|\\
\leq&\|A_{0}T_{\lambda}(x_{0})+v_0\| +\sum_{i=0}^{k}\delta_i
		<+\infty.
			\end{aligned}
	\]
Then, since the sequence $\{v_{k+1}\}$ is bounded, we obtain
	\[
		\|A_{k+1}T_{\lambda}(x_{k+1})\|
		 \leq \|A_{k+1}T_{\lambda}(x_{k+1})+v_{k+1}\|+\|v_{k+1}\| 		<+\infty,
			\]
			which leads to \eqref{B38}. Then, by the monotonicity of  $T_{\lambda}$, we can establish that
	\begin{equation}\nonumber
	\begin{aligned}
&	A_{k+1}\langle x_{k+1}-x, T_{\lambda}(x)\rangle\leq
	A_{k+1}\langle x_{k+1}-x, T_{\lambda}(x_{k+1})\rangle
\\\overset{\eqref{B16}}{\leq}&	\varTheta_{k+1}(x)\overset{\eqref{B17}}{\leq}	\varTheta_{k}(x)-
	\langle
x-	v_{k+1}, g_k
	\rangle\leq\varTheta_{0}(x)-\sum_{i=0}^{k}
	\langle
	x-v_{i+1}, g_i
	\rangle.
	\end{aligned}
	\end{equation}
	Subsequently, the inequality \eqref{B39} can be deduced from the definition of $\mathtt{Gap}(x)$ and the summablity of $\|g_k\|$.
\end{proof}

Before concluding this section, we establish the equivalence between the Halpern iteration and the Contracting PPA. The Halpern iteration was first proposed by Halpern \cite{halpern1967fixed}, and Lieder \cite{lieder2021convergence} established $O(1/k)$ convergence of $\|x_k-Nx_k\|$ compared to the $O(1/\sqrt{k})$ convergence of the Krasnosel'skii iteration \cite{Tworemarks}.
Recall that the Halpern iteration for solving non-expansive mappings is given by
\begin{equation}\label{Halpern}
	z_{k+1}=\frac{1}{k+2}z_0+\frac{k+1}{k+2}Nz_k,
\end{equation}
where $N:\mathcal{H}\longrightarrow {\mathcal{H}}$ is a non-expansive mapping satisfying
$$\|Nx-Ny\|\leq\|x-y\|,~~x,y\in\mathcal{H}.$$
The following theorem establishes the equivalence of the Halpern iteration \eqref{Halpern}  and the Contracting PPA under a special parameter setting.

\begin{theorem}\label{EQHAL}
Let  $\{z_k\}$ be generated by the Halpern iteration \eqref{Halpern} with $N=J_{T_{\lambda}}$. Then, $J_{T_{\lambda}}(	z_{k+1})$ is equivalent to using the
 Contracting PPA with $g_k\equiv0$, $a_k\equiv1$, $A_0=1$, and $A_{0}T_{\lambda}(x_{0})+v_0=z_0$.
\end{theorem}
\begin{proof}
Setting $N=J_{T_{\lambda}}$ in the Halpern iteration \eqref{Halpern} yields that
$$
	z_{k+1}=\frac{1}{k+2}z_0+\frac{k+1}{k+2}J_{T_{\lambda}}(z_k).
$$
Then, by projecting $J_{T_{\lambda}}$ onto $z_{k+1}$ and defining $x_{k+1}:=J_{T_{\lambda}}(	z_{k+1})$, we have
\begin{eqnarray}
J_{T_{\lambda}}(	z_{k+1})&=&J_{T_{\lambda}} \left(\frac{1}{k+2}z_0+\frac{k+1}{k+2}J_{T_{\lambda}}(z_k)\right)\nonumber\\\Longleftrightarrow
x_{k+1}&=&J_{T_{\lambda}} \left(\frac{1}{k+2}z_0+\frac{k+1}{k+2}x_k\right)\nonumber\\&=&\left(\frac{\lambda}{1+\lambda}I+\frac{1}{1+\lambda}J_{(1+\lambda)T}\right) \left(\frac{1}{k+2}z_0+\frac{k+1}{k+2}x_k\right).\label{B45}
\end{eqnarray}
Therefore,  \eqref{B45} is exactly equivalent to the  Contracting PPA \eqref{B42} with the parameter settings  $g_k\equiv0$, $a_k\equiv1$, $A_0=1$, and $A_{0}T_{\lambda}(x_{0})+v_0=z_0$.
\end{proof}

\section{Guaranteed sublinear convergence Contracting PPA}
We must emphasize that there is no guarantee of convergence for the Contracting PPA. In this section, we will focus on the sublinear ($O(1/k^p)$, $p\geq1$) convergence rate of the Contracting PPA with a guaranteed convergence. Instead of directly considering the convergence of the Contracting PPA, we will consider the following new method called {\bf Sublinear Contracting PPA}.

\begin{framed}
\noindent{\bf Sublinear Contracting PPA.}

\noindent
Given $\mu\in[0,1)$, $p\ge1$, $\tau_{-1}\in(0,1]$, $x_{-1}=x_0\in\mathcal{H}$, for $k=0,1,\cdots$, set
 \begin{eqnarray}
 \tau_k&=&\frac{1}{ak+1},~(0<a\leq \frac{1-\mu}{p}),\nonumber\\
			v_k&=&\frac{1}{\tau_{k-1}}x_k-\frac{1-\tau_{k-1}}{\tau_{k-1}}x_{k-1}.\label{B12v}
	\end{eqnarray}
Given $\lambda>0$, $\beta\ge0$,  and $\delta_k\ge0$ satisfying $\sum_{k\geq0}\delta_k<+\infty$,
find $x_{k+1}$ such that
\begin{equation}\label{B10}
0=\frac{1}{\tau_k^p}T_{\lambda}(x_{k+1})-\frac{1-\tau_k}{\tau_k^p}T_{\lambda}(x_{k})+(v_{k+1}-v_k)+\beta(x_{k+1}-x_k)+g_k,
\end{equation}
 where $g_k\in \mathcal{H}$  satisfies  that $\|g_k\|\leq\delta_k$.
\end{framed}

The sequence $\{\tau_k\}$ defined in the  Sublinear Contracting PPA has the following property.

\begin{lemma}\label{L8}
For $k\geq0$, it holds that $\tau_{k}\in(0,1]$ and  
	\begin{eqnarray}\label{B12}
	\frac{1-\tau_k}{\tau_k^p}&\leq&\frac{1}{\tau_{k-1}^p}-\frac{\mu}{\tau_k^{p-1}}.
	\end{eqnarray}
\end{lemma}
\begin{proof}
	The definition of $\tau_k$ implies that $\tau_k\in(0,1]$. When $k=0$, \eqref{B12} clearly holds.
	We can verify that, for $k\geq1$,
	\[
	\begin{aligned} \frac{1}{\tau_k^p}-\frac{1}{\tau_{k-1}^p}=&(ak+1)^p-\big(a(k-1)+1\big)^p\\
	=&ap\big(a\xi+1\big)^{p-1} ~~({\rm for~some~}\xi\in(k-1,k){\rm~by~the~Mean~Value ~theorem})\\\leq&ap\big(ak+1\big)^{p-1}=ap\frac{1}{\tau_k^{p-1}}\leq(1-\mu)\frac{1}{\tau_k^{p-1}}.
	\end{aligned}
	\]
	Then we can derive the conclusion.
\end{proof}


\begin{rem}\label{rem1}
	We turn our attention to the selection of $\tau_k$ in Sublinear Contracting PPA.
	If we set
	$
	\frac{1}{\tau_k}:=\frac{A_{k+1}}{a_{k+1}},
	$ with $A_{k}=a_{k}+A_{k-1}$ and $A_{k+1}^{p-1}=a_{k+1}^p$
	then we have
	\begin{equation}\label{tau:1}
	\frac{1}{\tau_k^p}=\frac{A_{k+1}^p}{a_{k+1}^p}=A_{k+1},~{\rm and} ~\frac{1-\tau_k}{\tau_k^p}=\frac{A_{k}A_{k+1}^{p-1}}{a_{k+1}^p}=A_k.
	\end{equation}
	Therefore, by setting
	$\beta=0$, we can express equation \eqref{B10} as
	\[
	0=A_{k+1}T_{\lambda}(x_{k+1})-A_kT_{\lambda}(x_{k})+\frac{A_{k+1}x_{k+1}-A_kx_k}{a_{k+1}}-\frac{A_{k}x_{k}-A_{k-1}x_{k-1}}{a_{k}}+g_k,
	\]
	which is equivalent to \eqref{B8} in the Contracting PPA with $A_{k+1}^{p-1}=a_{k+1}^p$.
	With this additional assumption, we can deduce from \eqref{tau:1} that
	\begin{equation}\label{tau:2}
	\frac{1-\tau_k}{\tau_k^p}=\left(A_k=\frac{A_{k}A_{k}^{p-1}}{a_{k}^p}=\frac{A_{k}^{p}}{a_{k}^p}=\right)\frac{1}{\tau_{k-1}^p}.
	\end{equation}
	So, if $p$ is either an integer greater than $4$ or not an integer, calculating $\tau_k$ from \eqref{tau:2} may not be straightforward, even if $\tau_{k-1}$ is given. Therefore, the inequality in condition \eqref{B12} is used for a more efficient calculation of $\tau_k$. Additionally, it should be noted that in order to establish the weak convergence of the iterations, we assume that $\mu>0$ and $\beta>0$, which are crucial assumptions.
\end{rem}

We introduce the following  potential function:
\begin{equation}\label{B13}
	\varTheta_{k+1}'(x):=\frac{1}{\tau_k^p}\langle x_{k+1}-x, T_{\lambda}(x_{k+1})\rangle+\frac{1}{2}\|v_{k+1}-x\|^2+\frac{\beta}{2}\|x_{k+1}-x\|^2,
\end{equation}
which will play a significant role in the analysis of the convergence rate and weak convergence of the Sublinear Contracting PPA.

\begin{lemma}\label{L7}
Suppose $x_*\in T^{-1}(0)$. The iterative sequence $\{x_k\}$ generated by  the Sublinear Contracting PPA  satisfies that
\begin{equation}\label{B14}
\begin{aligned}
		\varTheta_{k+1}'(x_*)\leq&	\varTheta_{k}'(x_*)-
		\left(
		\frac{1-\tau_k}{\tau_k^{p+1}}\langle x_{k+1}-x_k, T_{\lambda}(x_{k+1})-T_{\lambda}(x_{k}) \rangle
		+\frac{1}{2}\|v_{k+1}-v_k\|^2\right.\\&\left.+\frac{\mu}{\tau_k^{p-1}}\langle x_{k}-x_*, T_{\lambda}(x_{k})\rangle+\beta\frac{2-\tau_k}{2\tau_k}\|x_{k+1}-x_k\|^2+\langle
		x_*-v_{k+1}, g_k
		\rangle
		\right).
		\end{aligned}
\end{equation}
In addition,
\begin{itemize}
		\item[\textnormal{({1})}] $\{v_k\}$ is bounded.
	\item[\textnormal{({2})}]	$\sum_{k\geq0}\vert\langle
	v_{k+1}-x_*, g_k
	\rangle\vert<+\infty.$
		\item[\textnormal{({3})}]
$	\sum_{k\geq0}\frac{\mu}{\tau_k^{p-1}}\langle x_{k}-x_*, T_{\lambda}(x_{k})\rangle<+\infty$.
	\item[\textnormal{({4})}]
$	\sum_{k\geq0}	\frac{1-\tau_k}{\tau_k^{p+1}} \| T_{\lambda}(x_{k+1})-T_{\lambda}(x_{k})\|^2<+\infty$.
	\item[\textnormal{({5})}]
	$	\sum_{k\geq0} \|v_{k+1}-v_k\|^2<+\infty.$
	\item[\textnormal{({6})}]$\sum_{k\geq0}  	\beta\frac{2-\tau_k}{2\tau_k}\|x_{k+1}-x_k\|^2	<+\infty.$
\end{itemize}
\end{lemma}
We give the proof in Appendix \ref{APb} for better reading. The following theorem clarifies the sublinear convergence rate.

\begin{theorem}
If  $x_*\in T^{-1}(0)$, then the iterative sequence $\{x_k\}$ generated by the Sublinear Contracting PPA  satisfies that
\begin{eqnarray}
	\label{B34} 	\|T_{\lambda}(x_{k+1})\|\leq O(\tau_k^p)= O(1/k^p),\\
\label{B35}	\mathtt{Gap}(x_{k+1})\leq O(\tau_k^p)= O(1/k^p).
\end{eqnarray}
\end{theorem}
\begin{proof}
To begin, we will show that $\{x_k\}$ is bounded.
	Based on the definition of $v_{k+1}$, we have
	\begin{equation}\label{B47}
	\|v_{k+1}-x_*\|^2 =\frac{1-\tau_k}{\tau_k^2}\|x_{k+1}-x_k\|^2+\frac{1}{\tau_k}\|x_{k+1}-x_*\|^2-\frac{1-\tau_k}{\tau_k}\|x_{k}-x_*\|^2.
	\end{equation}
	Using the above equality and the definition of $\varTheta_{k+1}'(x)$ in \eqref{B13}, we can then derive \begin{equation}\label{B21}
		\begin{aligned}
\varTheta_{k+1}'(x_*)=&\frac{1}{\tau_k^p}\langle x_{k+1}-x_*, T_{\lambda}(x_{k+1})\rangle+\frac{1-\tau_k}{2\tau_k^2}\|x_{k+1}-x_k\|^2\\&+\frac{1+\beta\tau_k}{2\tau_k}\|x_{k+1}-x_*\|^2-\frac{1-\tau_k}{2\tau_k}\|x_{k}-x_*\|^2.
		\end{aligned}
	\end{equation}
Consequently, we have
\begin{equation}\label{B15}
	\begin{aligned} &\frac{1}{2\tau_k}\|x_{k+1}-x_*\|^2-\frac{1-\tau_k}{2\tau_k}\|x_{k}-x_*\|^2\overset{\eqref{B47}}{\leq}\frac{1}{2}\|v_{k+1}-x_*\|^2
\overset{\eqref{B13}}{\leq}\varTheta_{k+1}'(x_*)\\
\overset{\eqref{B14}}{\leq}&\varTheta_{k}'(x_*)-\langle
		x_*-v_{k+1}, g_k		\rangle\leq\varTheta_{0}'(x_*)-\sum_{i=0}^{k}\langle
		x_*-v_{i+1}, g_i
		\rangle<c,
			\end{aligned}
	\end{equation}
where $c$ is a positive constant and the last inequality follows from the second conclusion of in Lemma \ref{L7}.
Moreover, \eqref{B15} is equivalent to:
\begin{equation*}
	\frac{1}{2}\|x_{k+1}-x_*\|^2\leq\frac{1-\tau_k}{2}\|x_{k}-x_*\|^2+\tau_kc.
\end{equation*}
By utilizing Lemma \ref{L2} with $a_k=	\frac{1}{2}\|x_{k}-x_*\|^2$ and $\eta_k=\tau_k$, we conclude that $ \|x_{k}-x_*\|^2$ is  bounded.  As a result, $\{x_k\}$  is also bounded.
	
By rearranging the iteration  \eqref{B10} in the Sublinear Contracting PPA  as
\begin{equation*}
	\begin{aligned}
&\frac{1}{\tau_k^p}T_{\lambda}(x_{k+1})-\frac{1}{\tau_{k-1}^p}T_{\lambda}(x_{k})+\left(
\frac{1}{\tau_{k-1}^p}-\frac{1-\tau_k}{\tau_k^p}
\right)T_{\lambda}(x_{k})\\=&(v_{k}-v_{k+1})+\beta(x_{k}-x_{k+1})-g_k,
	\end{aligned}
\end{equation*}
and then summing them from $k=1$ to $k=t$, we obtain
	\begin{equation}\label{B18}
		\begin{aligned}
		&\frac{1}{\tau_t^p}T_{\lambda}(x_{t+1})-\frac{1}{\tau_{0}^p}T_{\lambda}(x_{1})+\sum_{k=1}^{t}\left(
		\frac{1}{\tau_{k-1}^p}-\frac{1-\tau_k}{\tau_k^p}
		\right)T_{\lambda}(x_{k})\\=&(v_{1}-v_{t+1})+\beta(x_{1}-x_{t+1})-\sum_{k=1}^{t}g_k.
		\end{aligned}
	\end{equation}
For convenience, let us define
$$
h_k:=
\frac{1}{\tau_{k-1}^p}-\frac{1-\tau_k}{\tau_k^p},~a'_k:=h_kT_{\lambda}(x_{k}),~{\rm and}~s_k':=\sum_{i=1}^{k}a_i'.
$$
Then \eqref{B18} is equivalent to:
\begin{equation}\label{B20}
\begin{aligned}
b_t':&=(v_{1}-v_{t+1})+\beta(x_{1}-x_{t+1})+\frac{1}{\tau_{0}^p}T_{\lambda}(x_{1})-\sum_{k=1}^{t}g_k\\&=	\underbrace{\frac{1}{h_{t+1}\tau_t^p}a_{t+1}'}_{\frac{1}{\tau_t^p}T_{\lambda}(x_{t+1})}+s_t'
=\frac{1}{h_{t+1}\tau_t^p}(s_{t+1}'-s_t')+s_t'
.
\end{aligned}
\end{equation}	
Due to the boundedness of $\{x_t\}$ and $\{v_t\}$, as well as the summability of $\|g_k\|$, we can conclude that  $\{b_t'\}$ is bounded. Furthermore, combining this result with \eqref{B20},  we can conclude that there exists a constant $C>0$ such that
 \begin{equation}\nonumber
s_{t+1}'\leq(1-h_{t+1}\tau_t^p)s_t'+h_{t+1}\tau_t^pb_t'\leq(1-h_{t+1}\tau_t^p)s_t'+h_{t+1}\tau_t^pC.
\end{equation}
By observing that
$$0\leq\frac{\mu\tau_k^p}{\tau_{k+1}^{p-1}}\overset{\eqref{B12}}{\leq}
\underbrace{\Big(\frac{1}{\tau_{k}^p}-\frac{1-\tau_{k+1}}{\tau_{k+1}^p}\Big)\tau_k^p}_{h_{k+1}\tau_k^p}<1,
$$
and then applying Lemma \ref{L2} with $a_k=s_k'$, $\eta_k=h_{t+1}\tau_t^p$, and $c=C$, we can conclude that $\{s_k'\}$ is bounded.
Additionally, noting that
$$
\frac{1}{\tau_t^p}\|T_{\lambda}(x_{t+1})\|=\|b_t'-s_t'\|<\|b_t'\|+\|s_t'\|<+\infty,
$$ we can prove \eqref{B34}.
Then according to
\begin{equation}\label{B59}
\begin{aligned}
 &\langle x_{k+1}-x, T_{\lambda}(x)\rangle    \leq  \langle x_{k+1}-x, T_{\lambda}(x_{k+1})\rangle \leq\|x_{k+1}-x\| \|T_{\lambda}(x_{k+1})\|,
 \end{aligned}
	\end{equation}
\eqref{B34}, the boundedness of $x_k$	 and the definition of $\mathtt{Gap}(x_{k+1})$, we can derive \eqref{B35}.
\end{proof}

The following theorem establishes the weak convergence of the Sublinear Contracting PPA under the conditions $\mu\in(0,1)$  and $\beta>0$.

\begin{theorem}
	Let $\{x_k\}$ be the sequence generated by  the Sublinear Contracting PPA  with $\mu\in(0,1)$  and $\beta>0$.  Then the sequence $\{x_k\}$ converges weakly to a solution  of \eqref{B}.
\end{theorem}
\begin{proof}
The proof is based on Lemma \ref{L4}. 	
	We first prove the first condition in  Lemma \ref{L4},  which states that $\|x_k-x_*\|^2$ is convergent.
		By substituting the reformulation of $\varTheta'_{k+1}(x_*)$,  given by \eqref{B21}, into \eqref{B14} and ignoring the  nonnegative terms $	
		\frac{1-\tau_k}{\tau_k^{p+1}}\langle x_{k+1}-x_k, T_{\lambda}(x_{k+1})-T_{\lambda}(x_{k}) \rangle
		+\frac{1}{2}\|v_{k+1}-v_k\|^2+\frac{\mu}{\tau_k^{p-1}}\langle x_{k}-x_*, T_{\lambda}(x_{k})\rangle+\beta\frac{2-\tau_k}{2\tau_k}\|x_{k+1}-x_k\|^2
		 $, we can infer that
\begin{footnotesize}
		\begin{equation}\nonumber
		\begin{aligned}
		&\frac{1}{\tau_k^p}\langle x_{k+1}-x_*, T_{\lambda}(x_{k+1})\rangle+\frac{1-\tau_k}{2\tau_k^2}\|x_{k+1}-x_k\|^2+\frac{1+\beta\tau_k}{2\tau_k}(\|x_{k+1}-x_*\|^2-\|x_{k}-x_*\|^2)\\\leq&
		\frac{1}{\tau_{k-1}^p}\langle x_{k}-x_*, T_{\lambda}(x_{k})\rangle+\frac{1-\tau_{k-1}}{2\tau_{k-1}^2}\|x_{k}-x_{k-1}\|^2+\frac{1-\tau_{k-1}}{2\tau_{k-1}}(\|x_{k}-x_*\|^2-\|x_{k-1}-x_*\|^2)\\&+\vert\langle
		v_{k+1}-x_*, g_k
		\rangle\vert.
		\end{aligned}
		\end{equation}
\end{footnotesize}
		Adding $\frac{\beta}{\tau_k^{p-1}}\langle x_{k+1}-x_*, T_{\lambda}(x_{k+1})\rangle+\beta\frac{1-\tau_k}{2\tau_k}\|x_{k+1}-x_{k}\|^2$ to both sides of the above inequality, we have
\begin{footnotesize}
		\[
		\begin{aligned}
		&\frac{1+\beta\tau_k}{\tau_k} (\frac{1}{\tau_k^{p-1}}\langle x_{k+1}-x_*, T_{\lambda}(x_{k+1})\rangle+\frac{1-\tau_k}{2\tau_k}\|x_{k+1}-x_k\|^2+\frac{1}{2}(\|x_{k+1}-x_*\|^2-\|x_{k}-x_*\|^2) )
\\ &\leq
		\frac{1}{\tau_{k-1}^p}\langle x_{k}-x_*, T_{\lambda}(x_{k})\rangle+\frac{1-\tau_{k-1}}{2\tau_{k-1}^2}\|x_{k}-x_{k-1}\|^2+\frac{1-\tau_{k-1}}{2\tau_{k-1}}(\|x_{k}-x_*\|^2-\|x_{k-1}-x_*\|^2)\\&~~~~+\vert\langle
		v_{k+1}-x_*, g_k
		\rangle\vert+\frac{\beta}{\tau_k^{p-1}}\langle x_{k+1}-x_*, T_{\lambda}(x_{k+1})\rangle+\beta\frac{1-\tau_k}{2\tau_k}\|x_{k+1}-x_{k}\|^2\\&=\frac{1-\tau_{k-1}}{\tau_{k-1}} (
		\frac{1}{\tau_{k-1}^{p-1}}\langle x_{k}-x_*, T_{\lambda}(x_{k})\rangle\!+\!\frac{1-\tau_{k-1}}{2\tau_{k-1}}\|x_{k}-x_{k-1}\|^2\!+\!\frac{1}{2}(\|x_{k}-x_*\|^2\!-\!\|x_{k-1}-x_*\|^2) )\\&~~~~+\vert\langle
		v_{k+1}-x_*, g_k
		\rangle\vert+\frac{\beta}{\tau_k^{p-1}}\langle x_{k+1}-x_*, T_{\lambda}(x_{k+1})\rangle+\beta\frac{1-\tau_k}{2\tau_k}\|x_{k+1}-x_{k}\|^2\\&~~~~+	\frac{1}{\tau_{k-1}^{p-1}}\langle x_{k}-x_*, T_{\lambda}(x_{k})\rangle+\frac{1-\tau_{k-1}}{2\tau_{k-1}}\|x_{k}-x_{k-1}\|^2.
		\end{aligned}
		\]
\end{footnotesize}
	According to Lemma \ref{L5} presented in the Appendix, by	setting \begin{equation*}
			\begin{aligned}
			&a_k:=\frac{1}{2}\|x_{k}-x_*\|^2,~ t_k:=\frac{1-\tau_{k-1}}{\tau_{k-1}},~ \eta:=\beta+1,\\&
			b_k:=\!	\frac{1}{\tau_{k-1}^{p-1}}\langle x_{k}\!-\!x_*, T_{\lambda}(x_{k})\rangle\!+\!\frac{1\!-\!\tau_{k-1}}{2\tau_{k-1}}\|x_{k}\!-\!x_{k-1}\|^2\!+\!\frac{1}{2}(\|x_{k}-x_*\|^2\!-\!\|x_{k-1}\!-\!x_*\|^2),\\&
			d_k:=\vert\langle
			v_{k+1}-x_*, g_k
			\rangle\vert+\frac{\beta}{\tau_k^{p-1}}\langle x_{k+1}-x_*, T_{\lambda}(x_{k+1})\rangle+\beta\frac{1-\tau_k}{2\tau_k}\|x_{k+1}-x_{k}\|^2\\&~~~~~~~+	\frac{1}{\tau_{k-1}^{p-1}}\langle x_{k}-x_*, T_{\lambda}(x_{k})\rangle+\frac{1-\tau_{k-1}}{2\tau_{k-1}}\|x_{k}-x_{k-1}\|^2,
			\end{aligned}
		\end{equation*}
we can verify that the first and the second inequalities in \eqref{A7} hold.
		Since $\tau_k=\frac{1}{ak+1}$ with $a>0$ and $ap\leq 1-\mu$, we have
		\begin{equation}\label{B53}
		\begin{aligned}
	&	\sum_{k=1}^{+\infty}	\frac{1}{\tau_{k-1}^{p-1}}\langle x_{k}-x_*, T_{\lambda}(x_{k})\rangle=	\sum_{k=1}^{+\infty}	\big(a(k-1)+1\big)^{p-1}\langle x_{k}-x_*, T_{\lambda}(x_{k})\rangle\\\leq&	\sum_{k=1}^{+\infty}	(ak+1)^{p-1}\langle x_{k}-x_*, T_{\lambda}(x_{k})\rangle=	\sum_{k=1}^{+\infty}	\frac{1}{\tau_{k}^{p-1}}\langle x_{k}-x_*, T_{\lambda}(x_{k})\rangle\overset{\text{Lemma}~ \ref{L7} ~	\textnormal{({3})}}{<}+\infty.
		\end{aligned}
		\end{equation}
Then, it follows from \eqref{B53}, conclusions \textnormal{({2})}, \textnormal{({3})}, and \textnormal{({6})} in \text{Lemma}~\ref{L7} that $d_k$ is summable. 	By using Lemma \ref{L5}, we can conclude that
 $\|x_k-x_*\|^2$ is convergent.
	
Next, we will prove the second condition in Lemma \ref{L4}.  Since  $\|T_{\lambda}(x_{k+1})\|\leq O(1/k^p)$, we have $\lim\limits_{k\rightarrow+\infty}\|T_{\lambda}(x_k)\|=0$. Therefore, every weak cluster point of $x_k$ belongs to a solution of \eqref{B} based on Lemma \ref{L6}.
\end{proof}

\section{Guaranteed linear convergence Contracting PPA}
In this section, we will focus on the linear convergence rate of the Contracting PPA with a guaranteed convergence. Instead of directly considering the convergence of the Contracting PPA, we will consider the following new method called {\bf Linear Contracting PPA}.

\begin{framed}
	\noindent{\bf Linear Contracting PPA.}
	
	\noindent
Given $\tau\in(0,1)$,  $x_{-1}=x_0\in\mathcal{H}$, for $k=0,1,\cdots$, set
	\begin{equation}\label{v:1}
	v_k=\frac{1}{\tau}x_k-\frac{1-\tau}{\tau}x_{k-1}.
	\end{equation}
Given  $\lambda>0$, $\eta\in(0,1]$, $\beta\geq0$,  $\delta_k\ge0$ with $\sum_{k\geq0}\delta_k<+\infty$, find $x_{k+1}$ such that
	\begin{equation}\label{B22} 0=\frac{1}{(1\!-\!\eta\tau)^{k+1}}T_{\lambda}(x_{k+1})-\frac{1\!-\!\tau}{(1\!-\!\eta\tau)^{k+1}}T_{\lambda}(x_{k})+(v_{k+1}-v_k)+\beta(x_{k+1}-x_k)+g_k,
	\end{equation}
 where $g_k\in \mathcal{H}$  satisfies  that $\|g_k\|\leq\delta_k$.
\end{framed}

\begin{rem}
{If we
set $A_{k}$ as in \eqref{B8v:0} with $A_0=1$ and $\frac{1}{\tau}:=\frac{A_{k}}{a_{k}}$, then $\frac{1-\tau}{\tau}=\frac{A_{k-1}}{a_{k}}$ and hence
	\begin{equation*}
		A_{k}=\frac{1}{\tau}a_{k} =\frac{1}{1-\tau}A_{k-1}=\frac{1}{(1-\tau)^{k}}A_0=\frac{1}{(1-\tau)^{k}}.
	\end{equation*} }
Then the Linear Contracting PPA with $\eta=1$ and $\beta=0$ is equivalent to the Contracting PPA. To establish the weak convergence of the Linear Contracting PPA, we further require
 $\eta\in(0,1)$ and $\beta>0$.
\end{rem}

In the convergence rate and weak convergence analysis of the Linear Contracting PPA, we consider the following  potential function
\begin{equation}\label{B23}
\varTheta_{k+1}''(x):=\frac{1}{(1-\eta\tau)^{k+1}}\langle x_{k+1}-x, T_{\lambda}(x_{k+1})\rangle+\frac{1}{2}\|v_{k+1}-x\|^2+\frac{\beta}{2}\|x_{k+1}-x\|^2,
\end{equation}
which plays a crucial role.

\begin{lemma}\label{L9}
	Suppose $x_*\in T^{-1}(0)$. The iterative sequence $\{x_k\}$ generated by the Linear Contracting PPA satisfies that
	\begin{equation}\label{B24}
	\begin{aligned}
	&\varTheta_{k+1}''(x_*)\\
=& \varTheta_{k}''(x_*)-
	\left(
	\frac{1-\tau}{\tau(1-\eta\tau)^{k+1}}\langle x_{k+1}-x_k, T_{\lambda}(x_{k+1})-T_{\lambda}(x_{k}) \rangle
	+\frac{1}{2}\|v_{k+1}-v_k\|^2\right.\\ &\left.+ 	\frac{\tau(1-\eta)}{(1-\eta\tau)^{k+1}}\langle x_{k}-x_*, T_{\lambda}(x_{k})\rangle+\beta\frac{2-\tau}{2\tau}\|x_{k+1}-x_k\|^2+\langle
	x_*-v_{k+1}, g_k
	\rangle
	\right).
	\end{aligned}
	\end{equation}
	In addition,
	\begin{itemize}
			\item[\textnormal{({1})}] $\{v_k\}$ is bounded.
	\item[\textnormal{({2})}]$\sum_{k\geq0}\vert\langle
	v_{k+1}-x_*, g_k
	\rangle\vert<+\infty$.
		\item[\textnormal{({3})}]$
\label{B61}	\sum_{k\geq0}\frac{\tau(1-\eta)}{(1-\eta\tau)^{k+1}}\langle x_{k}-x_*, T_{\lambda}(x_{k})\rangle<+\infty.$
	\item[\textnormal{({4})}]$
	\sum_{k\geq0}\frac{1-\tau}{\tau(1-\eta\tau)^{k+1}}\| T_{\lambda}(x_{k+1})-T_{\lambda}(x_{k})\|^2<+\infty$.
		\item[\textnormal{({5})}]
$	\sum_{k\geq0} \|v_{k+1}-v_k\|^2<+\infty.$
	\item[\textnormal{({6})}]
$	\sum_{k\geq0}  	\beta\frac{2-\tau}{2\tau}\|x_{k+1}-x_k\|^2	<+\infty.$
	\end{itemize}
\end{lemma}
 We give the proof in Appendic \ref{Apc} for better reading. The following theorem clarifies the linear convergence rate.

\begin{theorem}
	If $x_*\in T^{-1}(0)$,  the iterative sequence $\{x_k\}$ generated by the Linear Contracting PPA  satisfies that
	\begin{eqnarray}
	\label{B56}	\|T_{\lambda}(x_{k+1})\|\leq O\big((1-\eta\tau)^{k+1}\big);\\
\label{B57}\mathtt{Gap}(x_{k+1})\leq O\big((1-\eta\tau)^{k+1}\big).
	\end{eqnarray}
\end{theorem}
\begin{proof}
	First, we will prove that $\{x_k\}$ is bounded.	According to the definition of $v_{k+1}$, we have
	\begin{equation}\label{B54}
	\begin{aligned}
	&\|v_{k+1}-x_*\|^2
	=\frac{1-\tau}{\tau^2}\|x_{k+1}-x_k\|^2+\frac{1}{\tau}\|x_{k+1}-x_*\|^2-\frac{1-\tau}{\tau}\|x_{k}-x_*\|^2.
	\end{aligned}
	\end{equation}
	Then by using the above equality and  the definition of $\varTheta_{k+1}''(x)$ in \eqref{B23} we have \begin{equation}\label{B29}
	\begin{aligned}
	\varTheta_{k+1}''(x_*)=&\frac{1}{(1-\eta\tau)^{k+1}}\langle x_{k+1}-x_*, T_{\lambda}(x_{k+1})\rangle+\frac{1-\tau}{2\tau^2}\|x_{k+1}-x_k\|^2\\&+\frac{1+\beta\tau}{2\tau}\|x_{k+1}-x_*\|^2-\frac{1-\tau}{2\tau}\|x_{k}-x_*\|^2.
	\end{aligned}
	\end{equation}
Consequently, we have
\[
	\begin{aligned}	&\frac{1}{2\tau}\|x_{k+1}-x_*\|^2-\frac{1-\tau}{2\tau}\|x_{k}-x_*\|^2\overset{\eqref{B54}}{\leq}\frac{1}{2}\|v_{k+1}-x_*\|^2
\overset{\eqref{B23}}{\leq} \varTheta_{k+	1}''(x_*)\\ \overset{\eqref{B24}}{\leq}&\varTheta_{k}''(x_*)-\langle
	x_*-v_{k+1}, g_k
	\rangle
\leq\varTheta_{0}'(x_*)-\sum_{i=0}^{k}\langle
	x_*-v_{i+1}, g_i
	\rangle <c,
	\end{aligned}
\]	
where $c$ is a positive constant and the last inequality follows from the second conclusion of in Lemma \ref{L9}.
Moreover, we can rewrite the above inequality as
\begin{equation*}
\frac{1}{2}\|x_{k+1}-x_*\|^2\leq\frac{1-\tau}{2}\|x_{k}-x_*\|^2+
\tau c.
\end{equation*}
By utilizing Lemma \ref{L2} with $a_k=	\frac{1}{2}\|x_{k}-x_*\|^2$ and $\eta_k=\tau$,
we conclude that $ \|x_{k}-x_*\|^2$ bounded.  As a result, $\{x_k\}$  is also bounded.
	
By rearranging the iteration  \eqref{B22} in the Sublinear Contracting PPA  as
	\begin{equation*}
	\begin{aligned} &\frac{1}{(1-\eta\tau)^{k+1}}T_{\lambda}(x_{k+1})-\frac{1}{(1-\eta\tau)^{k}}T_{\lambda}(x_{k})+\frac{\tau(1-\eta)}{(1-\eta\tau)^{k+1}}T_{\lambda}(x_{k})\\
=&(v_{k}-v_{k+1})+\beta(x_{k}-x_{k+1})-g_k,
	\end{aligned}
	\end{equation*}
and then summing them from $k=1$ to $k=t$, we obtain
	\begin{equation}\label{B30}
	\begin{aligned} &\frac{1}{(1-\eta\tau)^{t+1}}T_{\lambda}(x_{t+1})-\frac{1}{(1-\eta\tau)^{1}}T_{\lambda}(x_{1})+\sum_{k=1}^{t}\frac{\tau(1-\eta)}{(1-\eta\tau)^{k+1}}T_{\lambda}(x_{k})\\
=&(v_{1}-v_{t+1})+\beta(x_{1}-x_{t+1})-\sum_{k=1}^{t}g_k.
	\end{aligned}
	\end{equation}
	For convenience, let us define
	$$
\theta:=\frac{\tau(1-\eta)}{1-\eta\tau}\in(0,1),~ 	a_k'':=\frac{\tau(1-\eta)}{(1-\eta\tau)^{k+1}}T_{\lambda}(x_{k}),~{\rm and}~s_k'':=\sum_{i=1}^{k}a_i''.
	$$
	Then, \eqref{B30} can be rewritten as
	\begin{equation}\label{B31}
	\begin{aligned}
	b_t'':&=(v_{1}-v_{t+1})+\beta(x_{1}-x_{t+1})+\frac{1}{1-\eta\tau }T_{\lambda}(x_{1})-\sum_{k=1}^{t}g_k\\&=	\frac{1}{\theta}a_{t+1}''+s_t''
	=\frac{1}{\theta}(s_{t+1}''-s_t'')+s_t''
	.
	\end{aligned}
	\end{equation}	
Due to the boundedness of $\{x_t\}$, $\{v_t\}$,  and the summablity of $\|g_k\|$,  $\{b_t''\}$ is also bounded. Together with \eqref{B31}, we can find $c>0$ such that
$$
	s_{t+1}''\leq(1-\theta)s_t''+\theta b_t''\leq(1-\theta)s_t''+\theta c.
$$
By using Lemma \ref{L2} with $a_k=	s_k''$ and $\eta_k=\theta$, we can conclude that $\{s_k''\}$ is bounded.
Additionally, noting that
	$$
\frac{1}{(1-\eta\tau)^{t+1}}\|T_{\lambda}(x_{t+1})\|	=\frac{1}{\theta}\|a_{t+1}''\|\overset{\eqref{B31}}{=}\|b_t''-s_t''\|<+\infty,
	$$ we can prove \eqref{B56}.
Then according to
\begin{equation}\nonumber
\begin{aligned}
&\langle x_{k+1}-x, T_{\lambda}(x)\rangle    \leq  \langle x_{k+1}-x, T_{\lambda}(x_{k+1})\rangle \leq\|x_{k+1}-x\| \|T_{\lambda}(x_{k+1})\|,
\end{aligned}
\end{equation}
\eqref{B56}, the boundness of $x_k$	 and the definition of $\mathtt{Gap}(x_{k+1})$, we can derive \eqref{B57}.The proof is complete.
\end{proof}

The following theorem establishes the weak convergence of the Linear Contracting PPA under the conditions  $\eta\in(0,1)$  and $\beta>0$.

\begin{theorem}
	Let $\{x_k\}$ be the sequence generated by the Linear Contracting PPA  with $\eta\in(0,1)$  and $\beta>0$.  The sequence $\{x_k\}$ converges  to a solution  of \eqref{B}.
\end{theorem}

\begin{proof}
	We will prove this theorem using Lemma \ref{L4}. Firstly, we will prove the first condition in Lemma \ref{L4}, which states that $\frac{1}{2}\|x_k-x_*\|^2$ is convergent.
	According to the  equivalent form of $\varTheta''_{k+1}(x_*)$ in \eqref{B29} and the relation \eqref{B24},  by ignoring the  nonnegative terms $	\frac{1-\tau}{\tau(1-\eta\tau)^{k+1}}\langle x_{k+1}-x_k, T_{\lambda}(x_{k+1})-T_{\lambda}(x_{k}) \rangle
	+\frac{1}{2}\|v_{k+1}-v_k\|^2+	\frac{\tau(1-\eta)}{(1-\eta\tau)^{k+1}}\langle x_{k}-x_*, T_{\lambda}(x_{k})\rangle+\beta\frac{2-\tau}{2\tau}\|x_{k+1}-x_k\|^2$ in \eqref{B24}, we can infer that
\[
	\begin{aligned}
	&\frac{1}{(1-\eta\tau)^{k+1}}\langle x_{k+1}-x_*, T_{\lambda}(x_{k+1})\rangle+\frac{1-\tau}{2\tau^2}\|x_{k+1}-x_k\|^2\\
&+\frac{1+\beta\tau}{2\tau}(\|x_{k+1}-x_*\|^2-\|x_{k}-x_*\|^2)\\
\leq&
	\frac{1}{(1-\eta\tau)^{k}}\langle x_{k}-x_*, T_{\lambda}(x_{k})\rangle+\frac{1-\tau}{2\tau^2}\|x_{k}-x_{k-1}\|^2\\
&+\frac{1-\tau}{2\tau}(\|x_{k}-x_*\|^2-\|x_{k-1}-x_*\|^2)
+\vert\langle
	x_*-v_{k+1}, g_k
	\rangle\vert.
	\end{aligned}
\]
	Adding $\frac{1+\beta\tau-\tau}{\tau(1-\eta\tau)^{k+1}}\langle x_{k+1}-x_*, T_{\lambda}(x_{k+1})\rangle+\beta\frac{1-\tau}{2\tau}\|x_{k+1}-x_{k}\|^2$ to both sides of the above inequality, we have
	\begin{footnotesize}
	\[
	\begin{aligned}
	&\frac{1+\beta\tau}{\tau}\Big(\frac{1}{(1-\eta\tau)^{k+1}}\langle x_{k+1}-x_*, T_{\lambda}(x_{k+1})\rangle+\frac{1-\tau}{2\tau}\|x_{k+1}-x_k\|^2\\
&+\frac{1}{2}(\|x_{k+1}-x_*\|^2-\|x_{k}-x_*\|^2)\Big)\\
\leq&
	\frac{1}{(1-\eta\tau)^{k}}\langle x_{k}-x_*, T_{\lambda}(x_{k})\rangle+\frac{1-\tau}{2\tau^2}\|x_{k}-x_{k-1}\|^2+\frac{1-\tau}{2\tau}(\|x_{k}-x_*\|^2-\|x_{k-1}-x_*\|^2)\\&+\vert\langle
	v_{k+1}-x_*, g_k
	\rangle\vert+\frac{1+\beta\tau-\tau}{\tau(1-\eta\tau)^{k+1}}\langle x_{k+1}-x_*, T_{\lambda}(x_{k+1})\rangle+\beta\frac{1-\tau}{2\tau}\|x_{k+1}-x_{k}\|^2\\=&\frac{1-\tau}{\tau}\Big(
\frac{1}{(1-\eta\tau)^{k}}\langle x_{k}-x_*, T_{\lambda}(x_{k})\rangle+\frac{1-\tau}{2\tau}\|x_{k}-x_{k-1}\|^2+\frac{1}{2}(\|x_{k}-x_*\|^2-\|x_{k-1}-x_*\|^2)\Big)\\&+\vert\langle
	v_{k+1}-x_*, g_k
	\rangle\vert+\frac{1+\beta\tau-\tau}{\tau(1-\eta\tau)^{k+1}}\langle x_{k+1}-x_*, T_{\lambda}(x_{k+1})\rangle+\beta\frac{1-\tau}{2\tau}\|x_{k+1}-x_{k}\|^2\\&+	\frac{2\tau-1}{\tau}
	\frac{1}{(1-\eta\tau)^{k}}\langle x_{k}-x_*, T_{\lambda}(x_{k})\rangle+\frac{1-\tau}{2\tau}\|x_{k}-x_{k-1}\|^2.
	\end{aligned}
	\]
	\end{footnotesize}	
	According to Lemma \ref{L5} presented in the Appendix, by	setting
\begin{equation*}
	\begin{aligned}
	&a_k:=\frac{1}{2}\|x_{k}-x_*\|^2,~ t_k:=\frac{1-\tau}{\tau},~ \eta:=\beta+1,\\&
	b_k:=	
	\frac{1}{(1\!-\!\eta\tau)^{k}}\langle x_{k}\!-\!x_*, T_{\lambda}(x_{k})\rangle\!+\!\frac{1\!-\!\tau}{2\tau}\|x_{k}\!-\!x_{k-1}\|^2\!+\!\frac{1}{2}(\|x_{k}\!-\!x_*\|^2\!-\!\|x_{k-1}\!-\!x_*\|^2),\\&
	d_k:=\vert\langle
	v_{k+1}-x_*, g_k
	\rangle\vert+\frac{1+\beta\tau-\tau}{\tau(1-\eta\tau)^{k+1}}\langle x_{k+1}-x_*, T_{\lambda}(x_{k+1})\rangle+\beta\frac{1-\tau}{2\tau}\|x_{k+1}-x_{k}\|^2\\&~~~~~~~+	\frac{2\tau-1}{\tau}
	\frac{1}{(1-\eta\tau)^{k}}\langle x_{k}-x_*, T_{\lambda}(x_{k})\rangle+\frac{1-\tau}{2\tau}\|x_{k}-x_{k-1}\|^2,
	\end{aligned}
	\end{equation*}

we can verify that the first and the second inequalities in \eqref{A7} hold.
Since $\eta\in(0,1)$ and $\beta>0$, conditions \textnormal{({2})},  \textnormal{({3})}, and    \textnormal{({6})} in Lemma \ref{L9} hold, along with the summability of $\vert\langle
v_{k+1}-x_*, g_k\rangle\vert$ (see \textnormal{({2})} in Lemma \ref{L9}),   we have that $\{d_k\}$ is summable. 	
Therefore, $\|x_k-x_*\|^2$ is convergent by using Lemma \ref{L5}.

Next, we will prove the second condition required in Lemma \ref{L4}.   It follows from $\|T_{\lambda}(x_{k+1})\|\leq  O\big((1-\eta\tau)^{k+1}\big)$ that $\lim\limits_{k\rightarrow+\infty}\|T_{\lambda}(x_k)\|=0$. Therefore, every weak cluster point of $x_k$ belongs to the solution of \eqref{B} by using Lemma \ref{L6}.
\end{proof}

\section{Conclusion}
To find a zero point of a maximally monotone operator $T$, instead of using PPA \eqref{PPA} to solve $T$, we propose employing PPA to solve its Yosida regularization $T_{\lambda}$.  Based on an $O(a_{k+1})$ ($a_{k+1}\geq \varepsilon>0$) resolvent index of $T$, we can achieve a convergence rate of
$O (1/{\sqrt{\sum_{i=0}^{k}a_{i+1}^2}} )$  for both  $\|T_{\lambda}(\cdot)\|$ and the gap function $\mathtt{Gap}(\cdot)$ in the non-ergodic sense. Moreover, we can achieve a convergence rate of  $O(1/\sum_{i=0}^{k}a_{i+1})$ for $\mathtt{Gap}(\cdot)$ in the ergodic sense. To improve the convergence rate of the newly-proposed PPA, we introduce an accelerated variant called  Contracting PPA. By leveraging a resolvent index of $T$ bounded by $O(a_{k+1})$ ($a_{k+1}\geq \varepsilon>0$), we establish a convergence rate of $O(1/\sum_{i=0}^{k}a_{i+1})$ for both $\|T_{\lambda}(\cdot)\|$ and
  $\mathtt {Gap}(\cdot)$, considering the non-ergodic sense.
It is important to note that the Contracting PPA does not provide a guarantee of convergence. Finally, we propose two new versions of the Contracting PPA with convergence guarantees. These updated algorithms have sublinear and linear convergence rates, respectively. One potential application of the Contracting PPA is to design Gauss-Seidel type accelerated algorithms with constant penalty parameters for solving the multi-block separable convex optimization problem with equality constraints:
\[
\min\limits_{x_i\in\mathcal{X}_i}\left\{\sum_{i=1}^{m}f_i(x_i):~ \sum_{i=1}^{m}A_ix_i=b\right\},
\]
where $m\geq1$, $f_i:\mathbb{R}^{n_i}\rightarrow\mathbb{R}$ is a closed proper convex function, and $\mathcal{X}_i\subseteq\mathbb{R}^{n_i}$ is a closed convex set.


\appendix

\section{Some auxiliary results}

In the Appendix, we present some lemmas used throughout this paper.

\begin{lemma}{\rm(\cite[Lemma A.8]{2017Convergence})}\label{L1}
	Let $c\geq0$ and $\{a_k\}$ and $\{b_k\}$ be nonnegative sequences (for $k\geq0$). If  $\{b_k\}$ is summable and $$
	a_k^2\leq c^2+\sum_{i=0}^{k}a_ib_i,
	$$then $a_k\leq c+\sum_{i=0}^{\infty}b_i$.
\end{lemma}

\begin{lemma}\label{L2}
Let $\{a_k\}$ and $\{\eta_k\}$ be nonnegative sequences (for $k\geq0$). If $$
a_{k+1}\leq  (1-\eta_k)a_k+\eta_kc,~c>0~{\rm and} ~\eta_k\in[0,1],
$$ then $a_k\leq \max \{a_0,c\}$ is bounded.
\end{lemma}
\begin{proof}
	We prove this conclusion by induction. First we have $a_0\leq \max \{a_0,c\}$. Suppose $a_i\leq \max \{a_0,c\}$ for $i\in\{0,1,2,\cdots,k\}$, then $$
	a_{k+1}\leq   (1-\eta_k)a_k+\eta_kc\leq  (1-\eta_k)\max \{a_0,c\}+\eta_kc\leq \max \{a_0,c\}.
	$$We obtain the conclusion.
\end{proof}

\begin{lemma}\label{L3}
	Suppose the nonnegative real sequences  $\{a_k\}$  and $\{d_k\}$  satisfy that $
	a_{k+1}\leq a_k +d_k$ and $\sum_{k=0}^{\infty}d_k<\infty$.
	Then $\{a_k\}$ is convergent.
\end{lemma}

The following lemma is used to establish the convergence of some given sequence. The proof is similar to \cite[Lemma 4.1]{2021Fast}.
\begin{lemma}\label{L5}
	Let $\{a_k\}\in\mathbb{R}$, $\{b_k\}\in\mathbb{R}$ and $\{d_k\}\in\mathbb{R}$ be real sequences such that $\{a_k\}$ and $\{d_k\}$ are nonnegative, and
	\begin{equation}\label{A7}
	a_{k+1}\leq a_k+b_{k+1},~~
	(t_{k+1}+\eta)b_{k+1}\leq t_k b_k+d_k ~~{\rm and}~~
	\sum_{k=0}^{\infty}d_k<+\infty,~~t_k,~\eta>0.
	\end{equation}
Then $\{a_k\}$ is convergent.
\end{lemma}
\begin{proof}
	Based on the second inequality given in \eqref{A7}, we have
	$$
(t_{k+1}+\eta)	b_{k+1}\leq t_k b_k+d_k \leq t_k [b_k]_++d_k,
	$$where $[\cdot]_+$ denotes the positive part. Because the right side of the above inequality is nonnegative, then we obtain {
	$$
(t_{k+1}+\eta)	[	b_{k+1}]_+\leq t_k [b_k]_++d_k.
	$$Then it gives
\[
\sum_{k=0}^{\infty} [b_k]_+\le\frac{1}{\eta}\sum_{k=0}^{\infty}(t_k [b_k]_+- t_{k+1} 	[	b_{k+1}]_++d_k)
<\frac{t_0 [b_0]_+}{\eta}+\frac{1}{\eta}\sum_{k=0}^{\infty}d_k <+\infty.
\]
}
According to the first inequality given in \eqref{A7}, it holds that
	$$
	a_{k+1}\leq a_k+b_{k+1}\leq a_k+[b_{k+1}]_+.
	$$By using Lemma \ref{L3}, we obtain that $\{a_k\}$ is convergent.
\end{proof}

The  following   lemma, known as Opial's Lemma, is used  to establish that the iteration sequences  converge weakly to a saddle point of \eqref{B}.
\begin{lemma}{\rm(\cite[Theorem 5.5]{Bauschke2011Convex})} \label{L4}
Let the sequence \textnormal{$\{x_k\}\in\mathcal{H}$} satisfies the following two conditions.
	\begin{itemize}
		\item[\textnormal{({1})}]  For  $x_*\in T^{-1}(0)$, $\lim\limits_{k\rightarrow\infty}\|x_k-x_*\|$ exists.
		\item[\textnormal{({2})}]Every weak cluster of  \textnormal{$\{x_k\}_{k\in\mathbb{N}}$} belongs to $T^{-1}(0)$.
	\end{itemize}
	Then  \textnormal{$\{x_k\}_{k\in\mathbb{N}}$} converges weakly to a point of $T^{-1}(0)$.
\end{lemma}

\begin{lemma}\label{L6}
For given sequence $\{x_k\}$,	if $ \lim\limits_{k\rightarrow+\infty}\|T_{\lambda}(x_k)\|=0$ and there exists a subsequence $\{x_{k_n}\}$ that converges to $\bar{x}$ as $n\rightarrow+\infty$, then $0\in T(\bar x)$.
\end{lemma}
\begin{proof}
	According to the given conditions, we have $ \lim\limits_{n\rightarrow+\infty}\|T_{\lambda}(x_{k_n})\|=0$, which means that $T_{\lambda}(x_{k_n})\in T(x_{k_n}-\lambda T_{\lambda}(x_{k_n}))$. By taking the limit as  $n\rightarrow+\infty$ and  considering the closed graph property of a maximally monotone operator, we can derive the desired conclusion.
\end{proof}

\section{Proof of Lemma \ref{L7}} \label{APb}
\begin{proof}
	Using basic algebraic operations, we can verify the following conclusions:
	\begin{small}
		\begin{eqnarray*}
			&&\begin{aligned}
				{\rm(i)}&\langle x_*-v_{k+1}, \frac{1}{\tau_k^p}T_{\lambda}(x_{k+1})\rangle
				=\langle x_*-x_{k+1}-\frac{1-\tau_k}{\tau_k}(x_{k+1}-x_k), \frac{1}{\tau_k^p}T_{\lambda}(x_{k+1})\rangle\\=&
				\frac{1}{\tau_k^p}\langle x_*-x_{k+1}, T_{\lambda}(x_{k+1})\rangle-\frac{1-\tau_k}{\tau_k^{p+1}}\langle
				x_{k+1}-x_k,T_{\lambda}(x_{k+1})
				\rangle;
			\end{aligned}\\
			&&\begin{aligned}
				{\rm(ii)}& \langle x_*-v_{k+1}, -\frac{1-\tau_k}{\tau_k^p}T_{\lambda}(x_{k}) \rangle
				= \langle x_*-x_{k}-\frac{1}{\tau_k}(x_{k+1}-x_k), -\frac{1-\tau_k}{\tau_k^p}T_{\lambda}(x_{k}) \rangle\\=&
				-\frac{1-\tau_k}{\tau_k^p}\langle x_*-x_{k}, T_{\lambda}(x_{k})\rangle+\frac{1-\tau_k}{\tau_k^{p+1}}\langle
				x_{k+1}-x_k,T_{\lambda}(x_{k})
				\rangle\\
				\overset{\eqref{B12}}{\leq}&- (\frac{1}{\tau_{k-1}^p}-\frac{\mu}{\tau_k^{p-1}} )\langle x_*-x_{k}, T_{\lambda}(x_{k})\rangle+\frac{1-\tau_k}{\tau_k^{p+1}}\langle
				x_{k+1}-x_k,T_{\lambda}(x_{k})
				\rangle;
			\end{aligned}\\
			&&{\rm(iii)}\langle x_*-v_{k+1}, v_{k+1}-v_k\rangle=-\frac{1}{2}(\|v_{k+1}-x_*\|^2-\|v_k-x_*\|^2+\|v_{k+1}-v_k\|^2);
			\\
			&&\begin{aligned}
				{\rm(iv)}	&\langle x_*-v_{k+1}, x_{k+1}-x_k\rangle=
				\langle   x_*-x_{k+1}-\frac{1-\tau_k}{\tau_k}(x_{k+1}-x_k),   x_{k+1}-x_k\rangle \rangle\\=&
				-\frac{1}{2}(\|x_{k+1}-x_*\|^2-\|x_k-x_*\|^2)-\frac{2-\tau_k}{2\tau_k}\|x_{k+1}-x_k\|^2.
			\end{aligned}
		\end{eqnarray*}
	\end{small}
	Adding the above four (in)equalities together yields the following condition:
	\begin{small}
		\begin{equation}\nonumber
		\begin{aligned}
		0=& \langle x_*-v_{k+1}, \frac{1}{\tau_k^p}T_{\lambda}(x_{k+1})-\frac{1-\tau_k}{\tau_k^p}T_{\lambda}(x_{k})+(v_{k+1}-v_k)+\beta(x_{k+1}-x_k)+g_k \rangle\\ \leq&\underbrace{\frac{1}{\tau_k^p}\langle x_*-x_{k+1}, T_{\lambda}(x_{k+1})\rangle-\frac{1}{2}\|v_{k+1}-x_*\|^2-\frac{\beta}{2}\|x_{k+1}-x_*\|^2}_{-\varTheta'_{k+1}(x_*)}
		\\&-\underbrace{ (\frac{1}{\tau_{k-1}^p}\langle x_*-x_{k}, T_{\lambda}(x_{k})\rangle-\frac{1}{2}\|v_{k}-x_*\|^2-\frac{\beta}{2}\|x_{k}-x_*\|^2 )}_{-\varTheta'_{k}(x_*)}\\&-
		\left(
		\frac{1-\tau_k}{\tau_k^{p+1}}\langle x_{k+1}-x_k, T_{\lambda}(x_{k+1})-T_{\lambda}(x_{k}) \rangle
		+\frac{1}{2}\|v_{k+1}-v_k\|^2\right.\\&\left.+\frac{\mu}{\tau_k^{p-1}}\langle x_{k}-x_*, T_{\lambda}(x_{k})\rangle+\beta\frac{2-\tau_k}{2\tau_k}\|x_{k+1}-x_k\|^2+	\langle
		x_*-v_{k+1}, g_k
		\rangle
		\right).
		\end{aligned}
		\end{equation}
	\end{small}
	\noindent Using some rearrangement of the inequality above, we obtain \eqref{B14}.
	The remaining conclusions are similar to those in Lemma \ref{LC}, so we will omit them here.
\end{proof}

\section{Proof of Lemma \ref{L9}}\label{Apc}
\begin{proof}
	Using basic algebraic operations, we can verify the following conclusions:
	\begin{eqnarray*}\label{B25}
		&&	\begin{aligned}
			{\rm(i)}& \langle x_*-v_{k+1}, \frac{1}{(1-\eta\tau)^{k+1}}T_{\lambda}(x_{k+1}) \rangle
			\\=& \langle x_*-x_{k+1}-\frac{1-\tau}{\tau}(x_{k+1}-x_k), \frac{1}{(1-\eta\tau)^{k+1}}T_{\lambda}(x_{k+1}) \rangle\\=&
			\frac{1}{(1-\eta\tau)^{k+1}}\langle x_*-x_{k+1}, T_{\lambda}(x_{k+1})\rangle-\frac{1-\tau}{\tau(1-\eta\tau)^{k+1}}\langle
			x_{k+1}-x_k,T_{\lambda}(x_{k+1})
			\rangle;
		\end{aligned}\\
		\label{B26}
		&&	\begin{aligned}
			{\rm(ii)}& \langle x_*-v_{k+1}, -\frac{1-\tau}{(1-\eta\tau)^{k+1}}T_{\lambda}(x_{k}) \rangle
			\\=& \langle x_*-x_{k}-\frac{1}{\tau}(x_{k+1}-x_k), -\frac{1-\tau}{(1-\eta\tau)^{k+1}}T_{\lambda}(x_{k}) \rangle\\=&
			-\frac{1-\tau}{(1-\eta\tau)^{k+1}}\langle x_*-x_{k}, T_{\lambda}(x_{k})\rangle+\frac{1-\tau}{\tau(1-\eta\tau)^{k+1}}\langle
			x_{k+1}-x_k,T_{\lambda}(x_{k})
			\rangle\\=&- (\frac{1}{(1-\eta\tau)^{k}}-\frac{\tau(1-\eta)}{(1-\eta\tau)^{k+1}} )\langle x_*-x_{k}, T_{\lambda}(x_{k})\rangle\\+&\frac{1-\tau}{\tau(1-\eta\tau)^{k+1}}\langle
			x_{k+1}-x_k,T_{\lambda}(x_{k})
			\rangle;
		\end{aligned} \\ \label{B27}
		&&{\rm(iii)}	\langle x_*-v_{k+1}, v_{k+1}-v_k\rangle=-\frac{1}{2}(\|v_{k+1}-x_*\|^2-\|v_k-x_*\|^2+\|v_{k+1}-v_k\|^2);\\
		\label{B28}
		&&\begin{aligned}
			{\rm(iv)}&\langle x_*-v_{k+1}, x_{k+1}-x_k\rangle=
			\langle   x_*-x_{k+1}-\frac{1-\tau}{\tau}(x_{k+1}-x_k),   x_{k+1}-x_k\rangle  \rangle\\=&
			-\frac{1}{2}(\|x_{k+1}-x_*\|^2-\|x_k-x_*\|^2)-\frac{2-\tau}{2\tau}\|x_{k+1}-x_k\|^2.
		\end{aligned}
	\end{eqnarray*}
	Adding the above four (in)equalities together yields the following condition:
	\[
	\begin{aligned}
	0=& \langle x_*-v_{k+1}, \frac{1}{(1-\eta\tau)^{k+1}}T_{\lambda}(x_{k+1})-\frac{1-\tau}{(1-\eta\tau)^{k+1}}T_{\lambda}(x_{k}) \\& +(v_{k+1}-v_k)+\beta(x_{k+1}-x_k)+g_k \rangle\\=&\underbrace{\frac{1}{(1-\eta\tau)^{k+1}}\langle x_*-x_{k+1}, T_{\lambda}(x_{k+1})\rangle-\frac{1}{2}\|v_{k+1}-x_*\|^2-\frac{\beta}{2}\|x_{k+1}-x_*\|^2}_{-\varTheta''_{k+1}(x_*)}
	\\&-\underbrace{ (\frac{1}{(1-\eta\tau)^{k}}\langle x_*-x_{k}, T_{\lambda}(x_{k})\rangle-\frac{1}{2}\|v_{k}-x_*\|^2-\frac{\beta}{2}\|x_{k}-x_*\|^2 )}_{-\varTheta''_{k}(x_*)}\\&-
	(
	\frac{1-\tau}{\tau(1-\eta\tau)^{k+1}}\langle x_{k+1}-x_k, T_{\lambda}(x_{k+1})-T_{\lambda}(x_{k}) \rangle
	+\frac{1}{2}\|v_{k+1}-v_k\|^2 \\& +\frac{\tau(1-\eta)}{(1-\eta\tau)^{k+1}}\langle x_{k}-x_*, T_{\lambda}(x_{k})\rangle+\beta\frac{2-\tau}{2\tau}\|x_{k+1}-x_k\|^2+	\langle
	x_*-v_{k+1}, g_k
	\rangle  ).
	\end{aligned}
	\]
	Then, by rearranging the above inequality, we can obtain \eqref{B24}. The remaining conclusions are similar to Lemma  \ref{LC}, so we will omit them.
\end{proof}

\section*{Statements and Declarations}
\begin{itemize}
\item Competing Interests\\
The authors declared that they have no competing interests to this work.
\item Availability of data and materials\\
There is no availability of supporting data and materials.
\end{itemize}

\bibliographystyle{plain}
\bibliography{ref}

\end{document}